\documentclass{birkjour}
\usepackage[all]{xy}

\usepackage{amssymb,amsthm,mathrsfs,xspace}
\mathcode`\:="603A     % : = \colon
\mathcode`\<="4268     % < = \langle
\mathcode`\>="5269     % > = \rangle
\mathchardef\gt="313E  % arithmetic
\mathchardef\lt="313C  % strict order
\mathchardef\colon="303A  % :=

\newtheorem{theorem}{Theorem}[section]
\newtheorem{lemma}[theorem]{Lemma}
\newtheorem{prop}[theorem]{Proposition}
\newtheorem{cor}[theorem]{Corollary}
\theoremstyle{definition}
\newtheorem{definition}[theorem]{Definition}
\newtheorem{remark}[theorem]{Remark}

\newtheorem{exm}[theorem]{Example}
\newtheorem{exms}[theorem]{Examples}

\edef\cdrestoreat{%%
\noexpand\catcode\lq\noexpand\@=\the\catcode\lq\@}\catcode\lq\@=11
\def\thmitem{\def\theenumi{\@roman\c@enumi}
\def\labelenumi{{\normalfont(\theenumi)}}}

\newenvironment{roll}{\list{}{\advance\leftmargini2\p@
 \itemindent-12\p@
 \labelwidth\z@ }}{\endlist}

\newenvironment{rolld}{\list{}{\advance\leftmargini2\p@
 \itemindent-12\p@
 \labelwidth\z@ }}{\endlist}

\cdrestoreat

\def\dfn#1{{\sl #1\/}}

\def\ct#1{\ensuremath{\mathscr{#1}}}
\def\Ct#1{\ensuremath{{\normalfont\textsf{\mdseries #1}}}}
\def\CM{\ensuremath{\mathcal{CM}}\xspace}
\def\SM{\ensuremath{\mathcal{SM}}\xspace}
\def\id#1{\ensuremath{\mathrm{id}_{#1}}}
\def\Id#1{\ensuremath{\mathrm{Id}_{#1}}}
\def\op{{}^{\textrm{\scriptsize op}}}
\def\opp{{}^{\textrm{\tiny op}}}
\def\Gr(#1){\ensuremath{\mathscr{G}_{#1}}}
\def\des#1{\ensuremath{\mathscr{D}\kern-.3ex\textit{es\kern.2ex}_{#1}}}
\let\tt\top

\let\Land\wedge

\let\Implies\Rightarrow

\let\ForalL\forall \def\Forall#1.{\ForalL_{#1}}
\let\ExistS\exists \def\Exists#1.{\ExistS_{#1}}
\def\LGE{\mathrel{\bgroup\ooalign{\hfil\raise.8ex\hbox{$\lt$}\hfil
 \crcr\hfil\raise-.5ex\hbox{$\eqslantgtr$}}\egroup}}
\def\Lge{\mathrel{\bgroup
 \ooalign{\hfil\raise.4ex\hbox{$\scriptscriptstyle\lt$}\hfil
 \crcr\hfil\raise-.3ex\hbox{$\scriptscriptstyle\eqslantgtr$}}\egroup}}
\def\lge{\mathchoice{\LGE}{\LGE}{\Lge}{\Lge}}
\def\gel{=}

\def\pr{\mathrm{pr}}
\def\cmp#1{\ensuremath{\{\kern-2.5pt|{#1}|\kern-2.5pt\}}}
\def\EH{\Ct{ED}\xspace}
\def\EE{\Ct{EED}\xspace}
\def\CH{\Ct{EqD}\xspace}
\def\QH{\Ct{QD}\xspace}
\def\Q#1{\ensuremath{\overline{#1}}}
\def\D{{\rotatebox[origin=c]{180}{\ensuremath{E}}\kern-.3ex}}
\def\B{{\rotatebox[origin=c]{180}{\ensuremath{A}}\kern-.6ex}}
\def\QD{\Q{{\rotatebox[origin=c]{180}{\ensuremath{E}}}}\kern-.3ex}
\def\start#1{\bigskip\noindent{\bfseries\itshape #1}}
\def\EC#1{\ensuremath
 {\left\lfloor\left.\kern-.2ex{#1}\kern-.2ex\right\rceil\right.}}
\def\Strut{\hbox{\vrule height.75em depth.35em width0pt}}
\def\ec#1{\ensuremath
 {\left\lfloor\left.\kern-.3ex\Strut{#1}\Strut\kern-.3ex\right\rceil\right.}}
\def\ds{\mathchoice{\textstyle\sum}{\sum}{\sum}{\sum}}
\def\dpd{\mathchoice{\textstyle\prod}{\prod}{\prod}{\prod}\kern-.2ex\strut}
\def\F{F^{ML}}
\def\Fm{G^{\mathbf{mtt}}}
\def\Fs{F^{\mathbf{mtt}}}
\def\emtt{\ensuremath{\mathbf{emtt}}\xspace}
\def\mtt{\ensuremath{\mathbf{mtt}}\xspace}
\def\CTT{\textrm{CoC}\xspace}
\def\prp{\kern.5ex\mbox{\normalfont\it prop}\kern.5ex}
\def\prps{\kern.5ex\mbox{\normalfont\it prop}_s\kern.5ex}
\def\ch#1{\ensuremath{\widehat{#1}}}
\def\blank{\mbox{--}}

\def\W#1#2{\ensuremath{\left[{#2}^{#1}\right]}}

\begin{document}

\title{Quotient completion for the foundation of constructive mathematics}
\author[M.E.~Maietti]{Maria Emilia Maietti}
\address{Dipartimento di Matematica Pura ed Applicata\\
Universit\`a di Padova\\
via Trieste 63\\
35121 Padova, Italy}
\email{maietti@math.unipd.it}
\author[G.~Rosolini]{Giuseppe Rosolini}
\address{DIMA, Universit\`a di Genova\\
via Dodecaneso 35\\
16146 Genova, Italy}
\email{rosolini@unige.it}
\thanks{Project MIUR-PRIN McTAFI provided
support for the research presented in the paper.}


\begin{abstract}
We apply some tools developed in categorical logic
to give an abstract description of constructions used to formalize
constructive mathematics in foundations based on intensional type
theory. 
The key concept we employ is that of a Lawvere hyperdoctrine for which
we describe a notion of quotient completion. That notion includes 
the exact completion on a category with weak finite limits as an
instance as well 
as examples from type theory that fall apart from this.
\end{abstract}

\maketitle

\section{Introduction}\label{intro}
Category theory provides a language to investigate 
the syntax and the semantics of formal systems on the same ground, as
it provides an appropriate abstraction  useful to bring to the
foreground an algebraic structure that usually remains hidden behind
both.
In fact, the present paper is a plain example of how category
theory offers a language which is suitable to describe a key property
that foundations of constructive mathematics 
should have according to \cite{mtt}.

In the following part of the Introduction, we address the relevance of
such abstract properties. After that we describe the
category-theoretic concepts that are dealt with in 
the paper. Finally, we examine two quotient models based on
intensional type theory. 

\paragraph*{The need of quotient completion to found constructive
 mathematics}
There are various foundations for constructive mathematics available
in the literature: some are formulated in axiomatic set theory, others
in category theory, yet others in type theory. In fact there is no
foundation for constructive mathematics as standard as
the theory ZFC is for classical mathematics.
The authors of \cite{mtt} propose to look for a minimalist foundation
which could form a common core for the most relevant constructive
theories. Its finalized construction is in \cite{m09}.

In {\it loc.cit.} the authors also state that a foundation for
constructive mathematics should make it evident which key aspects
differentiate it from classical mathematics. For instance,
contrary to classical proofs, 
constructive proofs enjoy the {\it existence property}, {\it i.e.} one
can extract programs that compute witnesses of existential statements
occurring in them. Even more, any proof of a constructive system 
should be seen as a program. Hence, ideally, a foundation for
constructive mathematics should be at the same time a set theory,
in which to formalize mathematical theorems, and a programming language
in which to extract the computational contents of mathematical
proofs. 

Type theory provides examples of such formal systems,
such as Martin-L{\"o}f's Intensional Type Theory~\cite{PMTT} or 
Coquand's Calculus of Constructions \cite{tc90}.
But there is a problem in adopting such type theories as a foundation
for constructive mathematics. First of all they do not validate
extensional features used in the everyday practice of mathematics
such as extensional equality of sets or of functions, nor do they
offer quotient constructions.  
Indeed, if one wants that these systems act as useful functional
programming languages, they must meet decidability issues on typing of
proofs which are incompatible with extensional features. This is argued
in more formal terms (with the notion of proofs-as-programs theory) in
\cite{mtt}.

The solution adopted in practice in the formalization of mathematics
in type theory is to represent extensional concepts in a model built
on top of the type theory, for example by using
setoids, see \cite{disttheshof,ven}. 
A main drawback of this approach for a constructive mathematician is
that working with setoids---and especially with dependent setoids---is
extremely complicated compared to adopting a foundation in the same
vein as ZFC for classical mathematics, like for instance Aczel's CZF
or Friedman's IZF. A natural solution is to work with an 
{\it axiomatization of the quotient model of setoids supported by the
intensional type theory} instead of working 
{\it directly in the model}.
In a sense, someone who would like to have a foundation for
constructive mathematics based on a type theory is naturally led to
abandon the traditional view of having a unique system to formalize
mathematics in favour of a two-level foundation where one level is
used for program extraction and the other to formalize mathematics.

This is the central idea of the  
{\it notion of constructive foundation} put forward 
in \cite{mtt}. There it is required that a foundation for
constructive mathematics should be {\it a two-level theory}: one
level, named {\it intensional}, should be used as a programming
language; the other, which is called {\it extensional}, should be
closed under standard extensional constructs in order to be used as
the actual formal system in which to 
perform mathematical proofs, and it should be seen as an abstraction
of the intensional level according to Sambin's forget-restore
principle, see \cite{toolbox}.  In \cite{m09} it was stated that to
satisfy the 
link between the two levels in \cite{mtt} it is enough  to interpret 
the extensional level in the intensional one {\it by means of a
quotient completion of the latter}, {\it i.e.} to see the extensional
level as a 
fragment of the internal language of a quotient completion built on 
the intensional one. The two-level minimalist foundation in \cite{m09}
provides an example of  such a constructive foundation.

What remains to specify in the notion of constructive foundation in \cite{m09}
is what one means ``abstractly'' by quotient completion.

In particular one wants to see whether the construction performed in
\cite{m09}---the quotient model built over the intensional level to
interpret the extensional level---is an instance of a free
construction on categories with structure.

In the literature on category theory various constructions of quotient
completion have been studied, for example  in
\cite{CarboniA:regec,RosoliniG:typtec,RosoliniG:loccce}.
These constructions all rely on defining quotients as stable effective
coequalizers of monic equivalence relations. Hence, they all produce
exact categories---and indeed they usually go under the name of
{\it exact completions}. 

But, as we observe in the present paper, 
the construction of quotients adopted in \cite{m09} does not
necessarily lead to an exact category and therefore it cannot be an
exact completion.
This motivates the quest for a more general notion of quotient
completion than the exact completion. 

In this paper we accomplish this task by relativizing
the notion of quotient to that of a suitable hyperdoctrine: the fibers 
act as the ``logic'' in which to consider equivalence relations. With
respect to those, we introduce a notion of 
quotient in the base category of the hyperdoctrine, and we prove that
that notion is algebraic. In other words, there is a universal
construction that ``freely adds'' quotients for the equivalence
relations without adding any further power to the logic.

We use a weakened notion of Lawvere 
hyperdoctrine~\cite{LawvereF:adjif,LawvereF:equhcs,LawvereF:diaacc,LawvereF:setfm},
here simply called ``elementary doctrine'', with respect to which
we present a universal construction of quotient completion, which we
call ``elementary quotient completion''.

Instances of this construction include both the quotient model in
\cite{m09} and the exact completion  of a category with finite limits 
in \cite{CarboniA:freecl}.
Indeed, the study of the elementary quotient completion helps to isolate,
in the doctrine setting,
the properties of a model ``with quotients'' and to handle those
properties independently from one another. 
Thanks to the more general setting than that, say, of categories with
(weak) finite limits, it is also possible to analyse the properties of
the elementary quotient completion that are similar to the exact
completion, such as closure under exponentials in
\cite{RosoliniG:loccce}.

As a biproduct of all this, we also obtain a clear
explanation of the well-known result that the category of total setoids 
\`a la Bishop built over Martin-L{\"o}f's type theory is the exact
completion of an appropriate category with weak finite limits, see
\cite{notepal,CarboniA:somfcr,RosoliniG:typtec}.

\section{Doctrines}
We introduce the notion of doctrine
that will be used to specify that of {\it quotient}. This notion
is an obvious generalization of that of a hyperdoctrine. 
Hyperdoctrines were introduced, in a series of seminal papers, by
F.W.~Lawvere to synthesize the structural properties of logical
systems, see 
\cite{LawvereF:adjif,LawvereF:equhcs,LawvereF:diaacc,LawvereF:setfm}. His
crucial intuition was to consider logical languages and theories as
indexed categories and to study their 2-categorical properties. For
instance, connectives and quantifiers are determined by
adjunctions. That approach proved to be extremely fruitful, see 
\cite{MakkaiM:firocl,LambekJ:inthoc,jacobbook,TaylorP:prafom,OostenJ:reaait} 
and the references therein.

Recall from \cite{LawvereF:adjif} that a hyperdoctrine is a
functor $F:\ct{C}\op\longrightarrow\Ct{Heyt}$
from a cartesian closed category $\ct{C}$ to the category of Heyting
algebras satisfying some further conditions: for every arrow 
$f:A\to B$ in $\ct{C}$, the homomorphism $F_f:F(B)\to F(A)$ of
Heyting algebras---$F_f$ denotes the action of the functor $F$ on the
arrow $f$---has a left adjoint $\D_{f}$ and a right adjoint $\B_{f}$
satisfying the Beck-Chevalley condition.

The intuition is that a hyperdoctrine determines an appropriate
categorical structure to abstract both notions of
first order theory and of interpretation.

A many-sorted first order theory gives rises directly to a
hyperdoctrine $F:\ct{C}\op\longrightarrow\Ct{Heyt}$---a detailed
presentation is given as example~\ref{lta}:
\begin{itemize}
\item the objects of $\ct{C}$ are declarations of sort variables
\item a morphism $f:A\to B$ of $\ct{C}$ is a list of terms of the
sorts in $B$, in the variables in the declaration $A$
\item an object $P$ in $F(A)$ is a property written with the variables
declared in $A$
\item a morphism $P\leq Q$ in $F(A)$ shows that
property $Q$ follows from the property $P$
\item a functor $F_f:F(B)\to F(A)$ represents the substitution, in
properties of the sort $B$, of the terms $f$ for the variables in $B$
\item the adjoints $\D_f$ and $\B_f$ represent the existential
quantifier and the universal quantifier, respectively.
\end{itemize}
Another instance of hyperdoctrine
$F:\ct{C}\op\longrightarrow\Ct{Heyt}$ is the following:
\begin{itemize}
\item the objects of $\ct{C}$ are sets
\item a morphism $f:A\to B$ of $\ct{C}$ is a function into the set
$B$ from the set $A$
\item an object $P$ in $F(A)$ is a subset of the set $A$
\item a morphism $P\leq Q$ in $F(A)$ indicates that $Q$ contains $P$
\item a functor $F_f:F(B)\to F(A)$ acts by inverse image along $f$ on
subsets of the set $B$
\item the adjoints $\D_f$ and $\B_f$ must be evaluated, on a subset
$P$ of $A$, respectively as
$\begin{array}{l}\D_f(P)=\left\{b\in B\mid 
\exists a\in A [b=f(a)\Land a\in P]\right\}\\
\B_f(P)=\left\{b\in B\mid 
\forall a\in A[b=f(a)\Implies a\in P]\right\}\end{array}$
\end{itemize}
Thus a model of the many-sorted first
order theory determines precisely a functor from the former
hyperdoctrine to the latter, as already pointed 
out in \cite{LawvereF:adjif,LawvereF:equhcs}.

Our aim is to take advantage of the algebraic presentation
of logic offered by hyperdoctrines and we shall consider a more
general notion with the structure needed to define a
quotient of an equivalence relation from that perspective. 
We shall follow \cite{LawvereF:equhcs} 
and use the word ``doctrine'' with some attribute
to christen the more general notion and others derived from it.
We shall be able to separate the logical components
producing a universal construction of completion by quotients of
doctrines.

Somehow reflecting the essential logical structure that is needed in
order to present the theory of an equivalence relation, the basic
concept is that of a contravariant functor from a category with finite
products into the category of inf-semilattices and inf-preserving
maps.

\begin{definition}\label{ph}
A \dfn{primary doctrine} is a functor 
$P:\ct{C}\op\longrightarrow\Ct{InfSL}$
from (the opposite of) a category \ct{C} with finite products
to the category of inf-semilattices, {\it i.e.} a contravariant functor 
$P:\ct{C}\op\longrightarrow\Ct{Pos}$ in the category of partial orders
$\Ct{Pos}$
such that
\begin{itemize}
\item for every object $A$ in \ct{C}, the
partial order $P(A)$ has finite infs
\item for every arrow $f:A\to B$ in \ct{C}, the monotone map
$P_f:P(B)\to P(A)$ preserves them.\footnote{Here and in the
sequel we write the action of a doctrine $P$ on an arrow as
$P_f$.}
\end{itemize}
\end{definition}

The structure of a primary doctrine is just what is needed
to handle a many-sorted logic with binary conjunctions and a true
constant, as seen in the following example. 

\begin{exm}\label{lta}
The leading logical example is the indexed order
$LT:\ct{V}\op\longrightarrow\Ct{InfSL}$ given by the
Lindenbaum-Tarski algebras of well-formed formulae of a first order
theory (with only one sort).

Given a theory $\mathcal{T}$ in a first order language
$\mathcal{L}$, the domain category of the functor
is the category \ct{V} of lists of variables and term substitutions:
\begin{roll}
\item[object of \ct{V}] are lists\footnote{The empty list is included.}
of distinct variables $\vec x=(x_1,\ldots,x_n)$
\item[arrows] are lists of substitutions\footnote{We shall
employ a vector notation for lists of terms in the language as well as
for simultaneous substitutions such as $[\vec t/\vec y]$ in place of
$[t_1/y_1,\ldots,t_m/y_m]$. We shall also extend vectorial notation to
conjunctions and quantifiers writing $\vec t=\vec s$ for the
conjunction $t_1=s_1\Land\ldots\Land t_\ell=s_\ell$, provided the lists
$\vec t$ and $\vec s$ are the same length, and writing $\Exists\vec x.$
instead of $\Exists x_1.\ldots\Exists x_n.$.} for variables
$[\vec t/\vec y]:\vec x\to \vec y$ where each term $t_j$ in $\vec t$
is built in $\mathcal{L}$ on the variables $x_1,\ldots,x_n$
\item[composition] 
$\xymatrix@1@=4em{\vec x\ar[r]^{[\vec t/\vec y]}&
\vec y\ar[r]^{[\vec s/\vec z]}&\vec z}$
is given by simultaneous substitutions
$$\xymatrix@=10em{\vec x
\ar[r]^{\left[s_1[\vec t/\vec y]/z_1,\ldots,s_k[\vec t/\vec y]/z_k\right]}
&\vec z}$$
\end{roll}

The product of two objects $\vec x$ and $\vec y$ is given by a(ny)
list $\vec w$ of as many distinct variables as the sum of the number
of variables in $\vec x$ and of that in $\vec y$. Projections are
given by substitution of the variables in $\vec x$ with the
first in $\vec w$ and of the variables in $\vec y$ with the last in
$\vec w$.

The functor $LT:\ct{V}\op\longrightarrow\Ct{InfSL}$ is given
as follows: for a list of distinct variables $\vec x$,
the category $LT(\vec x)$ has
\begin{roll}
\item[objects] equivalence classes $\ec{W\kern.2ex}$ of well-formed
formulae $W$  of 
$\mathcal{L}$ with no more free variables than $x_1$,\ldots,$x_n$ with
respect to provable reciprocal consequence 
$W\dashv\vdash_{\mathcal{T}}W'$ in $\mathcal{T}$.\footnote{We shall
denote an equivalence class with representative $x$ as
$\ec{x\kern.2ex}$ in order to leave plain square brackets available for
other situations.} 
\item[arrows] $\ec{W\kern.2ex}\to\ec{V\kern.2ex}$ are
the provable consequences $W\vdash_{\mathcal{T}}V$ in $\mathcal{T}$ for
some pair of representatives (hence for any pair)
\item[composition] is given by the cut rule in the logical calculus
\item[identities] $\ec{W\kern.2ex}\to\ec{W\kern.2ex}$
are given by the logical rules $W\vdash_{\mathcal{T}}W$
\end{roll}
Observe that, in particular, for a list of distinct variables $\vec x$, the category
$LT(\vec x)$ has finite limits: products are given by conjunctions of
formulae and a terminal object is any provable formula, such as
$\vec x=\vec x$, that is any formula
equivalent to the true constant.
\end{exm}

\begin{exm}\label{model}
The following example of primary doctrine 
$S:\ct{S}\op\longrightarrow\Ct{InfSL}$ is the set-theoretic hyperdoctrine
described in the introduction and it can be considered in
an(y) axiomatic set theory such as ZF. We briefly recall its definition:
\begin{itemize}
\item $\ct{S}$ is the category of sets and functions,
\item $S(A)$ is the poset category of subsets of the set $A$
whose morphisms are inclusions,
\item a functor $S_f:S(B)\to S(A)$ acts as the inverse image $f^{-1}U $
on a subset $U$ of the set $B$.
\end{itemize}
\end{exm}

The example~\ref{lta} suggests that, by considering only 
doctrines, from a logical point of view one restricts attention
to the mere existence of a proof of a consequence, {\it i.e.} 
one only deals with proof irrelevance.

As already pointed out in \cite{LawvereF:adjif,LawvereF:equhcs}, a
set-theoretic model of a first order theory determines precisely a
functor from the doctrine $LT$ to the doctrine $S$ that
preserves all the structure of a primary doctrine.

As the example \ref{model},
also the example \ref{lta} gives rise to a Lawvere hyperdoctrine
when performed on a many-sorted first order
theory giving rise to a cartesian closed  base category. And the
characterization of set-theoretic model extends directly, see 
{\it loc.cit.}

\begin{remark}\label{wow2}
In many senses it is more general---and more elegant---to treat the
abstract theory of the relevant structures for the present paper in
terms of fibrations. For instance, a different, but equivalent
presentation of the 
structure above is as a faithful fibration $p:\ct{A}\to\ct{C}$
between categories with binary products such that $p$ preserves
them and has a right inverse right adjoint.

In fact, a primary doctrine
$P:\ct{C}\op\longrightarrow\Ct{InfSL}$
determines a faithful fibration $p_P:\Gr(P)\to\ct{C}$
by a well-known, general construction due to
Grothendieck, see \cite{GrothendieckA:catfd,jacobbook}, which applies
to indexed categories. We recall very
briefly that construction in the present situation. 
The data for the \dfn{total} category $\Gr(P)$ of $P$ are as follows:
\begin{roll}
\item[objects of $\Gr(P)$] are pairs  $(A,\alpha)$ such that 
$A$ is an object in \ct{C} and $\alpha$ is an object in $P(A)$.
\item[an arrow $(f,\phi):(A,\alpha)\to(B,\beta)$] is a pair of an
arrow $f:A\to B$ in \ct{C} and an arrow $\phi:\alpha\to P_f(\beta)$
\item[composition] of $(f,\phi):(A,\alpha)\to(B,\beta)$ and
$(g,\psi):(B,\beta)\to(C,\gamma)$ is 
$(g\circ f,P_f(\psi)\circ\phi)$. 
\end{roll}
One checks that a product of objects $(A,\alpha)$ and $(B,\beta)$ is
given by 
$$\xymatrix@C=3em{(A,\alpha)&(C,\chi)\ar[l]_{(\pr_1,\pi_1)}
\ar[r]^{(\pr_2,\pi_2)}&(B,\beta)}$$ 
where 
$$\xymatrix@C=3em{A&C\ar[l]_{\pr_1}\ar[r]^{\pr_2}&B}$$
is a product in \ct{C} and 
$$\xymatrix@C=3em{P_{\pr_1}(\alpha)&
\chi\ar[l]_(.35){\pi_1}\ar[r]^(.35){\pi_2}&
P_{\pr_2}(\beta)}$$
is a product in $P(C)$.

The first projection extends to a functor $p_F:\Gr(P)\to\ct{C}$ which 
is easily seen to be faithful with a right inverse right adjoint.

On the other hand, given a 
faithful fibration $p:\ct{A}\to\ct{C}$, one considers the functor
$\ch{p}:\ct{C}\op\longrightarrow\Ct{InfSL}$ which maps an object $A$
in \ct{C} to the partial order which is the poset reflection of the
preorder of the vertical arrows on $A$, see \cite{jacobbook}, 
{\it i.e.} one first considers the subcategory $p^A$ of \ct{A} 
consisting of those objects $\alpha$ such that $p(\alpha)=A$ and a map
$g:\alpha\to\alpha'$ of \ct{A} is in $p^A$ if $p(g)=\id{A}$;
faithfulness of $p$ ensures that the category $p^A$ is a
preorder. Product preservation ensures that $p^A$ has binary products,
the right inverse right adjoint ensures that $p^A$ has a terminal
object. So the poset reflection of $p^A$ produces the partial order 
$\ch{p}(A)$ on the equivalence classes of objects of $p^A$ with
respect to the equivalence given by isomorphism, where
$\ec{\alpha}\leq\ec{\alpha'}$ if there is an arrow
$g:\alpha\to\alpha'$ in $p^A$ for some pair of representatives (hence
for any pair), and the partial order has finite infs.

For an arrow $f:B\to A$ in \ct{C}, the functor
$\ch{p}_f\, :\ch{p}(A)\to\ch{p}(B)$ sends an equivalence class
$\ec{\alpha}$ to the equivalence class $\ec{\beta}$ such that there
is a cartesian lifting $g:\beta\to\alpha$ of $f$.

Setting up an appropriate 2-category for each structure (one for
primary doctrines, one for faithful fibrations as above), it is
easy to see that the two constructions extend to an equivalence
between those 2-categories.

Computing the total category of each of the examples of doctrines in
\ref{lta} and \ref{model}, one gets the following.

The total category $\Gr(LT)$ is the syntactic presentation of the
$\mathcal{L}$-definable subsets of (the finite powers of) a set
underlying a model of the theory $\mathcal{T}$ with functions defined
by terms in $\mathcal{L}$.

The total category $\Gr(S)$ is the full subfibration on subset
inclusions of the codomain fibration
$\textrm{cod}:\ct{S}^{\to}\longrightarrow\ct{S}$ on the category of
sets and functions.
\end{remark}

\begin{definition}\label{elh}
A primary doctrine 
$P:\ct{C}\op\longrightarrow\Ct{InfSL}$
is \dfn{elementary}
if, for every $A$ and $C$ in \ct{C},
the functor $P_{id_C\times \Delta_A}: P(C\times (A\times A))\to
P(C\times A)$\footnote{We write $\Delta_A$ and $f\times f'$ 
respectively for the map $<\id{A},\id{A}>$ and for the map 
$<f\circ\pr_1,f'\circ\pr_2>:A\times A'\to B\times B'$,
provided $f:A\to B$ and $f':A'\to B'$.}
has a left adjoint 
$\D_{id_C\times \Delta_A}$, and these satisfy
\begin{rolld}
\item[\dfn{Frobenius reciprocity}:] for every $A$ and $C$ in \ct{C},
for $\alpha$ in $P(C\times (A\times A))$, $\beta$ in $P(C\times A)$,
the canonical arrow 
$$\D_{id_C\times \Delta_A}(P_{id_C\times \Delta_A}(\alpha)
\Land_{C\times A}\beta)\leq
\alpha\Land_{C\times (A\times A)}\D_{id_C\times \Delta_A}(\beta)$$
in $P(C\times (A\times A))$ is iso (hence an identity).
\item [\dfn{Beck-Chevalley condition}:]
for any pullback diagram
$$\xymatrix{B\times A\ar[r]^{\quad id_B\times \Delta_A\qquad  }\ar[d]_{f\times id_A}& B\times A\times A
\ar[d]^{f\times id_{A \times A} }\\C\times A\ar[r]^{\quad id_C\times \Delta_A\qquad  }&C\times A\times A}$$
the canonical arrow 
$\D_{id_B\times \Delta_A}P_{f\times id_A}(\alpha)\leq P_{f\times id_{A \times A}} 
\D_{id_C\times \Delta_A}(\alpha)$  is iso in $P(B\times A\times A)$ for any
$\alpha$ in $P(C\times A)$.
\end{rolld}
\end{definition}
We refer the reader to \cite{LawvereF:equhcs,jacobbook} for a thorough
analysis of the concepts in the definition just given.

\begin{remark}
\label{equ}
For an elementary doctrine 
$P:\ct{C}\op\longrightarrow\Ct{InfSL}$, for any object $A$ in
\ct{C}, taking $C$ a terminal object, the conditions in \ref{elh} 
ensure the existence of a left adjoint $\D_{\Delta_A}$ to 
$P_{\Delta_A}:P(A\times A)\to P(A)$. On an object $\alpha$ in
$P(A)$ it can be written as 
\begin{equation}\label{delta}
\D_{\Delta_A}(\alpha)=P_{\pr_1}(\alpha)\Land_{A\times A}\D_{\Delta_A}(\tt_A)=
P_{\pr_2}(\alpha)\Land_{A\times A}\D_{\Delta_A}(\tt_A).
\end{equation}
where $\tt_A$ is the terminal object of $P(A)$.

Because of (\ref{delta}), an abbreviation like $\delta_A$ for the
object $\D_{\Delta_A}(\tt_A)$ is useful.
\end{remark}

\begin{exm}\label{ltae}
For $\mathcal{T}$ a first order theory, the primary fibration
$LT:\ct{V}\op\longrightarrow\Ct{InfSL}$, as defined in 
\ref{lta}, is elementary exactly when $\mathcal{T}$ has an equality
predicate.
\end{exm}

\begin{exm}\label{monoe}
The standard example of an elementary doctrine is the fibration of
subobjects. Consider a category \ct{X} with
finite products and pullbacks, {\it i.e.} with right adjoints to the
diagonal functors $\ct{X}\to\ct{X}^2$ and 
$\ct{X}\to
\ct{X}^{\strut\smash{\begin{array}
{@{}r@{}}\scriptstyle\downarrow\\[-1.8ex]\scriptstyle\to\cdot\kern.2ex
\end{array}}}$.
The functor $S:\ct{X}\op\longrightarrow\Ct{InfSL}$ assigns to any
object $A$ in \ct{X} the poset $S(A)$ whose
\begin{roll}
\item[objects] are subobjects 
$\ec{\smash{\xymatrix@1{\alpha:X\ \ar@{>->}[r]&A}}}$
in \ct{X} with codomain $A$
\item[$\ec{\alpha}\leq\ec{\alpha'}$] if there is
a commutative diagram
$$\xymatrix@C=1em{
X\strut\ar@{>->}[rd]_{\alpha}\ar[rr]^{x}&&X'\strut\ar@{>->}[ld]^{\alpha'}\\
&A}$$
for a (necessarily unique) arrow $x:X\to X'$.
\end{roll}
For an arrow $f:B\to A$, the assignment 
mapping an equivalence $\ec{\alpha}$ in $S(A)$ to that represented
by the left-hand arrow in the (chosen) pullback
$$\xymatrix{Y\ar[d]_{\beta}\ar[r]&X\ar[d]^{\alpha}\\B\ar[r]_f&A}$$
produces a functor $S_f:S(A)\to S(B)$ which preserves products.

Post-composition with an equalizer provides the elementary structure
since equalizers are monic.

In example~\ref{monoe}, we used the same notation for the functor $S$
as in example~\ref{model} because that is a particular instance of
\ref{monoe} when $\ct{X}$ is the category $\ct{S}$ of sets and
functions since each subobject
$\ec{\smash{\xymatrix@1{\alpha:X\ \ar@{>->}[r]&A}}}$ 
has a unique inclusion 
$\xymatrix@1{U\ar@<-.3ex>@{^(->}[r]&A}$ among its
representatives.
\end{exm}

\begin{exm}\label{weaklhy}
Consider a cartesian category \ct{C} with 
weak pullbacks,
{\it i.e.} for every pair $f:B\to A$, $g:C\to A$ of arrows in \ct{C},
there is a commutative diagram
$$\xymatrix@=4em{V_{f,g}\ar[d]_{\pr_{2,f,g}}\ar[r]^{\pr_{1,f,g}}&
 C\ar[d]^{g}\\B\ar[r]_{f}&A}$$
such that, for any $s:T\to B$, $t:T\to C$ satisfying 
$f\circ s=g\circ t$, there is $u:T\to V_{f,g}$ such that 
$t=\pr_{1,f,g}\circ u$ and $s=\pr_{2,f,g}\circ u$.

One can consider the functor $\Psi:\ct{C}\op\longrightarrow\Ct{InfSL}$
given by the poset reflection of each comma category $\ct{C}/A$ as
$A$ varies over the objects of \ct{C}. So, for a
given object $A$ in \ct{C}, one first considers the preorder whose
\begin{roll}
\item[objects] are arrows $\alpha:X\to A$ in \ct{C} with codomain $A$
\item[$\alpha\leq\alpha'$] if there is
an arrow $x:X\to X'$ in \ct{C} providing a commutative diagram
$$\xymatrix@C=1em{X\ar[rd]_{\alpha}\ar[rr]^{x}&&X'\ar[ld]^{\alpha'}\\&A}$$
\end{roll}
Then quotients that preorder with respect to the reciprocal
relation $\alpha\lge\alpha'$ to obtain the poset $\Psi(A)$.

For an arrow $f:B\to A$, the assignment that maps an equivalence class
$\ec{\alpha}$ in $\Psi(A)$ to that represented by the left-hand arrow
$\pr_{2,f,\alpha}:V_{f,\alpha}\to B$ in the weak (chosen) pullback 
$$\xymatrix@=4em{V_{f,\alpha}\ar[d]_{\pr_{2,f,\alpha}}\ar[r]^{\pr_{1,f,\alpha}}&
 C\ar[d]^{\alpha}\\B\ar[r]_{f}&A}$$
produces a functor $\Psi_f:\Psi(A)\to\Psi(B)$. It is easy to check
that it preserves products---which are given by weak pullbacks over
$A$.

The elementary structure is given by post-composition with the
equalizer.

This example is a slight generalization of a similar one given in
\cite{LawvereF:equhcs}.
\end{exm}

\begin{remark}
Note that the apparently minor difference between the example in 
\ref{monoe} and that in \ref{weaklhy} depends crucially on the
possibility of factoring an arbitrary arrow as a retraction 
followed by a monomorphism: for instance, in the category \ct{S} of
sets and functions, the fact that the two doctrines are equivalent can
be achieved thanks to the Axiom of Choice. 
\end{remark}

Consider the 2-category \EH of elementary doctrines:
\begin{roll}
\item[a 1-arrow] from $P$ to $R$ is a pair $(F,b)$ where
$F:\ct{C}\to\ct{D}$ is a functor which preserves finite products and
$b$ is a natural transformation
$$
\xymatrix@C=4em@R=1em{
{\ct{C}\op}\ar[rd]^(.4){P}_(.4){}="P"\ar[dd]_{F\op}&\\
           & {\Ct{InfSL}}\\
{\ct{D}\op}\ar[ru]_(.4){R}^(.4){}="R"&\ar"P";"R"_b^{\kern-.4ex\cdot}}
$$
such that for every object
$A$ in \ct{C}, the functor $b_A:P(A)\to R(F(A))$ preserves all the
structure. More explicitly, $b_A$ preserves finite meets and, for every
object $A$ in \ct{C}, 
\begin{equation}\label{two}
b_{A\times A}(\delta_A)\gel R_{<F(\pr_1),F(\pr_2)>}(\delta_{F(A)}).
\end{equation}
\item[a 2-arrow] $\theta:(F,b)\to(G,c)$ is a natural transformation
$\xymatrix@1@C=1.5em{\theta:F\ar[r]^(.6){.}&G}$ such that in the
diagram 
$$
\xymatrix@C=9em@R=1em{
{\ct{C}\op}\ar[rd]^(.45){P}_(.45){}="P"
\ar@<-1ex>@/_/[dd]_{F\opp}^{}="F"\ar@<1ex>@/^/[dd]^{G\opp}_{}="G"&\\
           & {\Ct{InfSL}}\\
{\ct{D}\op}\ar[ru]_(.45){R}^(.45){}="R"&
\ar@/_/"P";"R"_{b\kern.5ex\cdot\kern-.5ex}="b"
\ar@<1ex>@/^/"P";"R"^{\kern-.5ex\cdot\kern.5ex c}="c"
\ar"G";"F"_{.}^{\theta\opp}\ar@{}"b";"c"|{}}
$$
it is $b_A(\alpha)\leq R_{\theta_A}(c_A(\alpha))$ for every object $A$
in \ct{C} and every $\alpha$ in $P(A)$.
\end{roll}

The following definition is similar to \ref{elh} and contribute the
final part of the essential ``logical'' structure of an indexed
poset.

\begin{definition}\label{exh}
A primary doctrine 
$P:\ct{C}\op\longrightarrow\Ct{InfSL}$
is \dfn{existential}
if, for $A_1$ and $A_2$ in \ct{C}, for a(ny) projection
$\pr:A_1\times A_2\to A_i$, $i=1,2$,
the functor $P_{\pr_i}:P(A_i)\to P(A_1\times A_2)$ has a left adjoint
$\D_{\pr_i}$, to which we shall unimaginatively refer as
\dfn{existential}, and these satisfy
\begin{rolld}
\item[\dfn{Beck-Chevalley condition}:] for any pullback diagram
$$\xymatrix{X'\ar[r]^{\pr'}\ar[d]_{f'}&A'\ar[d]^f\\X\ar[r]^{\pr}&A}$$
with $\pr$ a projection (hence also $\pr'$ a projection), for any
$\beta$ in $P(X)$, the canonical arrow 
$\D_{\pr'}P_{f'}(\beta)\leq P_f\D_\pr(\beta)$ in $P(A')$ is iso; 
\item[\dfn{Frobenius reciprocity}:] for $\pr:X\to A$ a projection,
$\alpha$ in $P(A)$, $\beta$ in $P(X)$, the canonical arrow
$\D_\pr(P_\pr(\alpha)\Land_X\beta)\leq\alpha\Land_A\D_\pr(\beta)$ in
$P(A)$ is iso.
\end{rolld}
\end{definition}
About this notion we refer the reader to \cite{LawvereF:adjif,jacobbook}.

\begin{exms}
\noindent(a)
The primary fibration $LT:\ct{V}\op\longrightarrow\Ct{InfSL}$, as
defined in \ref{lta} for a first order theory $\mathcal{T}$, is
existential. An existential 
left adjoint to $P_\pr$ is computed by quantifying existentially the
variables that are not involved in the substitution given by the
projection, {\it e.g.} for the projection $\pr=[x/z]:(x,y)\to(z)$
and a formula $W$ with free variables at most $x$ and $y$, $\D_\pr(W)$
is $\Exists y.(W[z/x])$.

We stop to note that the example reveals the meaning of the
Beck-Chevalley condition: suppose $S$ and $T$ are sorts and consider
the morphism $[x/z]:(x,y)\to(z)$. On a formula $W$ with free variables
at most $x$ and $y$, 
for any morphism $[t/z]:(w_1,\ldots,w_n)\to(z)$,
the diagram
$$\xymatrix@C=7.5em{
(w_1,\ldots,w_n,y)
\ar[d]_{[t/x,y]}\ar[r]^(.51){[w_1/w_1,\ldots,w_n/w_n]}&
(w_1,\ldots,w_n)\ar[d]^{[t/z]}\\
(x,y)\ar[r]^(.55){[x/z]}&(z)}$$
is a pullback and the Beck-Chevalley condition rewrites the fact that
substitution commutes with quantification as
$$
\Exists y.(W[t/x])\equiv(\Exists y.W[z/x])[t/z]
$$
since the declaration $(w_1,\ldots,w_n)$ ensures that
$y$ does not appear in $t$.

\noindent(b)
For a cartesian category \ct{C} with weak pullbacks, 
the elementary doctrine $\Psi:\ct{C}\op\longrightarrow\Ct{InfSL}$
given in \ref{weaklhy} is existential. Existential left adjoints
are given by post-composition.

\noindent(c)
The primary doctrine in example~\ref{model} is existential: on a subset
$P$ of $A$, the adjoint $\D_\pr$, for a projection $\pr:A\to B$, must
be evaluated as 
$\D_\pr(P)=\left\{b\in B\mid 
\exists a\in A [a\in \pr^{-1}\left\{b\right\}\cap P]\right\}$, usually
called the \dfn{image of $P$ along} $\pr$.
\end{exms}

\begin{remark}\label{here}
In an existential elementary doctrine, for every map $f:A\to B$
in \ct{C} the functor $P_f$ has a left adjoint $\D_f$ that can be
computed as 
 $$\D_{\pr_2}(P_{f\times \id B}(\delta_B)\wedge P_{\pr_1}(\alpha))$$
for $\alpha$ in $P(A)$, where $\pr_1$ and $\pr_2$ are the projections
from $A\times B$.
\end{remark}

\begin{exm}
For a category \ct{X} with products and pullbacks, the elementary
doctrine $S:\ct{X}\op\longrightarrow\Ct{InfSL}$ in \ref{monoe} is
existential if and only if \ct{X} has a stable proper factorization 
system $(\mathcal{E},\mathcal{M})$, 
see \cite{HughesJ:facsft,PavlovicD:mapsii}. So, in particular, for
\ct{X} regular, the subobject doctrine
$S:\ct{X}\op\longrightarrow\Ct{InfSL}$ is elementary existential.
\end{exm}

Consider the 2-full 2-subcategory \EE of \EH whose objects are
elementary existential doctrines.
\begin{roll}
\item[The 1-arrows] are those pairs $(F,b)$ in \EH
such that $b$ preserves the
left adjoints along projections.
\end{roll}

Hence, by \ref{here}, the second functor of a 1-arrow in \EE
preserves  left adjoints along all arrows in $\ct{C}$.

\section{Quotients in an elementary doctrine}

The structure of elementary doctrine is suitable to describe the
notion of an equivalence relation and that of a quotient for such a
relation.

\begin{definition}\label{per}
Given an elementary doctrine
$P:\ct{C}\op\longrightarrow\Ct{InfSL}$, an object $A$ in
\ct{C} and an object $\rho$ in $P(A\times A)$, we say that 
$\rho$ is a \dfn{$P$-equivalence relation on $A$} if
it satisfies
\begin{rolld}
\item[\dfn{reflexivity}:] $\delta_A\leq\rho$
\item[\dfn{symmetry}:]
$\rho\leq P_{<\pr_2,\pr_1>}(\rho)$, for $\pr_1,\pr_2:A\times A\to A$
the first and second projection, respectively
\item[\dfn{transitivity}:]
$P_{<\pr_1,\pr_2>}(\rho)\Land P_{<\pr_2,\pr_3>}(\rho)\leq
P_{<\pr_1,\pr_3>}(\rho)$, for $\pr_1,\pr_2,\pr_3:A\times A\times A\to A$
the projections to the first, second and third factor,
respectively.
\end{rolld}
\end{definition}

\begin{exms}\label{exer}

\noindent(a)
Given an elementary doctrine
$P:\ct{C}\op\longrightarrow\Ct{InfSL}$ and an object $A$ in \ct{C},
the object $\delta_A$ is a $P$-equivalence relation on $A$.

\noindent(b)
Given a first order theory $\mathcal{T}$ with equality predicate,
consider the elementary doctrine 
$LT:\ct{V}\op\longrightarrow\Ct{InfSL}$ as in~\ref{ltae}. An
$LT$-equivalence relation is a
$\mathcal{T}$-provable equivalence relation.

\noindent(c)
For a category \ct{X} with products and pullbacks, consider the
elementary doctrine of subobjects
$S:\ct{X}\op\longrightarrow\Ct{InfSL}$ as in \ref{monoe}. An
$S$-equivalence relation is an equivalence relation in the category
\ct{X}.

\noindent(d)
For a cartesian category \ct{C} with weak pullbacks, consider the
elementary doctrine $\Psi:\ct{C}\op\longrightarrow\Ct{InfSL}$. A
$\Psi$-equivalence relation is a pseudo-equivalence relation in
\ct{C}, see \cite{CarboniA:freecl}.
\end{exms}

\begin{remark}\label{reme}
Let $P:\ct{C}\op\longrightarrow\Ct{InfSL}$ be
an elementary doctrine.
For an arrow $f:A\to B$ in \ct{C}, 
the functor $P_{f\times f}:P(B\times B)\to P(A\times A)$ takes a
$P$-equivalence relation $\sigma$ on $B$ to a $P$-equivalence relation
on $A$.
\end{remark}

\begin{definition}
Let $P:\ct{C}\op\longrightarrow\Ct{InfSL}$ be an elementary
doctrine. Let $\rho$ be a $P$-equivalence relation on $A$. 

A \dfn{quotient of $\rho$} is an arrow $q:A\to C$ in \ct{C} such that
$\rho\leq P_{q\times q}(\delta_C)$ 
and, for every arrow $g:A\to Z$ such that
$\rho\leq P_{g\times g}(\delta_Z)$, there is a unique arrow $h:C\to Z$
such that $g=h\circ q$. 
Such a quotient is \dfn{stable} when, for every arrow $f:C'\to C$ in
\ct{C}, there is a pullback 
$$\xymatrix{A'\ar[d]_{f'}\ar[r]^{q'}&C'\ar[d]^{f}\\A\ar[r]_q&C}$$
in \ct{C} and the arrow $q':A'\to C'$ is a quotient of 
the $P$-equivalence relation $P_{f'\times f'}(\rho)$.

Let $f:A\to B$ be an arrow in \ct{C}. The 
\dfn{$P$-kernel of $f:A\to B$}
is the $P$-equivalence relation $P_{f\times f}(\delta_B)$.
A quotient $q:A\to B$ of the $P$-equivalence relation $\rho$ is
\dfn{effective} if its $P$-kernel is $\rho$.
\end{definition}

\begin{exms}

\noindent(a)
Given an elementary doctrine
$P:\ct{C}\op\longrightarrow\Ct{InfSL}$ and an object $A$ in \ct{C},
a quotient of the $P$-equivalence relation $\delta_A$ on $A$ is the
identity arrow $\id{A}:A\to A$. It is trivially stable and effective
by definition.

\noindent(b)
In the elementary doctrine $S:\ct{X}\op\longrightarrow\Ct{InfSL}$ 
obtained from a category \ct{X} with products and pullbacks, a
quotient of the $S$-e\-quiv\-a\-lence relation
$\ec{\smash{\xymatrix@=2.3ex@1{r:R\ \ar@{>->}[r]&A\times A}}}$ is 
precisely a coequalizer of the pair of 
$$\xymatrix@C=5em{R\ar@<.5ex>[r]^(.45){\pr_1\circ r}
\ar@<-.5ex>[r]_(.45){\pr_2\circ r}&A}$$
---hence of any such pair obtained from the class
$\ec{\smash{\xymatrix@=2.3ex@1{r:R\ \ar@{>->}[r]&A\times A}}}$.
In particular, all $S$-equivalence relations have quotients which
are stable and effective if and only if the category \ct{C} is exact.
\end{exms}

\section{Set-like doctrines}

We intend to develop doctrines that may interpret constructive
theories for mathematics. We shall address two crucial properties that 
an elementary doctrine should verify in order to sustain such
interpretations. One relates to the axiom of comprehension and to
equality, the other to quotients.

\begin{definition}\label{compd}
Let $P:\ct{C}\op\longrightarrow\Ct{InfSL}$ be a primary
doctrine. Let $A$ be an object in \ct{C} and $\alpha$ an object
in $P(A)$.

A \dfn{comprehension of $\alpha$} is an arrow
$\cmp\alpha:X\to A$ in \ct{C} such that 
$\tt_{X}\leq P_{\cmp\alpha}(\alpha)$ and, 
for every arrow $g:Y\to A$ such that $\tt_{Y}\leq P_g(\alpha)$ there is
a unique $h:Y\to X$ such that $g=\cmp\alpha\circ h$.
Such a comprehension is \dfn{stable} when, for
every arrow $f:A'\to A$ in \ct{C}, $P_{f}(\alpha)$ has a
comprehension.

We say that $P:\ct{C}\op\longrightarrow\Ct{InfSL}$
\dfn{has comprehensions} if, for every object $A$ in
\ct{C}, every $\alpha$ in $P(A)$ has a comprehension.
\end{definition}

Again we refer the reader to \cite{LawvereF:equhcs}. 

The primary doctrine $S:\ct{S}\op\longrightarrow\Ct{InfSL}$ of
\ref{model} has comprehensions given by the trivial remark that a
subset determines an actual function by inclusion.

Among the examples listed in \ref{exer}, only example
(c), the doctrine of subobjects $S:\ct{X}\op\longrightarrow\Ct{InfSL}$ 
for \ct{X} a category with 
products and pullbacks, has comprehensions.

Example (d), the elementary doctrine
$\Psi:\ct{C}\op\longrightarrow\Ct{InfSL}$ constructed as in 
\ref{weaklhy} for a cartesian category 
\ct{C} with weak pullbacks, suggests to modify
the requirements in \ref{compd} by dropping uniqueness of the
mediating arrows. Before doing that, we note the following.

\begin{remark}\label{str}
For $f:A'\to A$ in \ct{C}, the mediating arrow $f'$ between the
comprehensions $\cmp\alpha:X\to A$ and $\cmp{P_{f}(\alpha)}:X'\to A'$
produces a pullback
$$\xymatrix@=3em{X'\ar[d]_{f'}\ar[r]^{\cmp{P_{f}(\alpha)}}&A'\ar[d]^{f}\\
X\ar[r]_{\cmp\alpha}&A.}$$
Hence a primary doctrine with comprehensions has comprehensions
stable under pullbacks.
\end{remark}

\begin{definition}\label{wcd}
Let $P:\ct{C}\op\longrightarrow\Ct{InfSL}$ be a primary
doctrine. Let $A$ be an object in \ct{C} and $\alpha$ an object
in $P(A)$.

A \dfn{weak comprehension of $\alpha$} is an arrow
$\cmp\alpha:X\to A$ in \ct{C} such that 
$\tt_{X}\leq P_{\cmp\alpha}(\alpha)$ and, 
for every arrow $g:Y\to A$ such that $\tt_{Y}\leq P_g(\alpha)$ there is
a (not necessarily unique) $h:Y\to X$ such that 
$g=\cmp\alpha\circ h$.\footnote{When necessary to distinguish between
the notions in definitions~\ref{compd} and \ref{wcd}, we shall refer
to one as in \ref{compd} with the further attribute \dfn{strict}.}

Such a comprehension is \dfn{stable}  when, for
every arrow $f:A'\to A$ in \ct{C}, $P_{f}(\alpha)$ has a weak
comprehension and there is a weak pullback
$$\xymatrix@=3em{X'\ar[d]_{f'}\ar[r]^{\cmp{P_{f}(\alpha)}}&A'\ar[d]^{f}\\
X\ar[r]_{\cmp\alpha}&A.}$$

We say that 
$P:\ct{C}\op\longrightarrow\Ct{InfSL}$,
\dfn{has weak comprehensions} if, for every object $A$ in
\ct{C}, every $\alpha$ in $P(A)$ has a weak comprehension.
\end{definition}

\begin{remark}\label{cmpm}
Suppose $\cmp\alpha:X\to A$ is a weak comprehension of $\alpha$.
The arrow $\cmp\alpha$ is monic if and only if it is a strict
comprehension.

Thus, when a diagonal is a weak comprehension, the arrow itself
satisfies the condition strictly (as in \ref{compd}). But some of its
reindexings may satisfy the weaker condition without uniqueness.
\end{remark}

\begin{exm}
For a cartesian category \ct{C} with weak pullbacks, 
the elementary doctrine $\Psi:\ct{C}\op\longrightarrow\Ct{InfSL}$ as
in \ref{weaklhy} has weak comprehensions.
\end{exm}

\begin{prop}\label{dswc}
Suppose $P:\ct{C}\op\longrightarrow\Ct{InfSL}$ is an elementary
doctrine. If the diagonal arrow $\Delta_A:A\to A\times A$ is a
stable $($weak$)$ comprehension of $\delta_A$, then for every pair of 
parallel arrows 
$\xymatrix@1{X\ar@<.5ex>[r]^f\ar@<-.5ex>[r]_g&A}$ in \ct{C}, the
$($weak$)$ comprehension of $P_{<f,g>}(\delta_A)$ is a $($weak$)$ 
equalizer of $f$ and $g$.
\end{prop}

\begin{proof}
It follows immediately from the construction of a weak equalizer
of $\xymatrix@1{X\ar@<.5ex>[r]^f\ar@<-.5ex>[r]_g&A}$ as a
weak pullback of $\Delta_A$: in the diagram
$$\xymatrix@=3em{
E\ar[d]_{f\circ\cmp{P_{<f,g>}(\delta_A)}}\ar[r]^{\cmp{P_{<f,g>}(\delta_A)}}
&X\ar[d]_{<f,g>}\ar@<1ex>[rd]^f\ar@<.1ex>[rd]_g\\
A\ar[r]_(.45){\Delta_A}&
A\times A\ar@<.5ex>[r]^{\pr_1}\ar@<-.5ex>[r]_{\pr_2}&\strut A&}$$
the square is a weak pullback. Since the bottom horizontal arrow
is the equalizer of the parallel pair that follows it, the top
horizontal arrow is a weak equalizer.\end{proof}

We say that an elementary doctrine
$P:\ct{C}\op\longrightarrow\Ct{InfSL}$,
has \dfn{comprehensive $($weak\/$)$ equalizers} if, for
every object $A$ in \ct{C}, the diagonal $\Delta_A:A\to A\times A$ is
a stable (weak) comprehension of $\delta_A$.

\begin{definition}
Let $P:\ct{C}\op\longrightarrow\Ct{InfSL}$ be a primary
doctrine. Let $A$ be an object in \ct{C} and $\alpha$ an object
in $P(A)$.

A (weak) comprehension $\cmp\alpha:X\to A$ of $\alpha$ is \dfn{full}
if $\alpha\leq_A\beta$ whenever $\tt_X\leq_X P_{\cmp\alpha}(\beta)$
for $\beta$ in $P(A)$.
\end{definition}

Note that the notion of full (weak) comprehension ensures that
$\alpha\leq_A\beta$ is equivalent to 
$\tt_X\leq_X P_{\cmp\alpha}(\beta)$ for $\beta$ in $P(A)$.

In an elementary doctrine it follows directly from the definition of
$\delta_A$ that the diagonal arrow $\Delta_A:A\to A\times A$ is a full
comprehension if and only if it is the comprehension of $\delta_A$.

\begin{cor}\label{dm}
Suppose that $P:\ct{C}\op\longrightarrow\Ct{InfSL}$ is an elementary
doctrine with 
full comprehensions and comprehensive equalizers.
If $f:A\to B$ is monic, then $P_{f\times f}(\delta_B)=\delta_A$.
\end{cor}

\begin{proof}
It follows by fullness of comprehensions after noting that
the comprehension of each side
is the kernel of $f$, which in turn follows 
from~\ref{dswc} and the fact that the kernel  of a monic
is the diagonal.\end{proof}

The next lemma will be needed in section \ref{appl}.

\begin{lemma}\label{wfs}
Let $P:\ct{C}\op\longrightarrow\Ct{InfSL}$ be a primary doctrine
with full weak comprehensions. Suppose also that, for a given $\alpha$
in $P(A)$, a weak comprehension $\cmp\alpha:X\to A$ is such that the
functor $P_{\cmp\alpha}:P(A)\to P(X)$ has a right adjoint 
$\B_{\cmp\alpha}:P(X)\to P(A)$. Then $\alpha\Land\blank:P(A)\to P(A)$ 
has a right adjoint $\alpha\Implies\blank:P(A)\to P(A)$.
\end{lemma}

\begin{proof}
Consider 
$\alpha\Implies\beta\colon=\B_{\cmp\alpha}(P_{\cmp\alpha}(\beta))$. To
see that  $\gamma \leq \B_{\cmp\alpha}(P_{\cmp\alpha}(\beta))$
if and only if
$\alpha\Land\gamma\leq\beta$ proceed as follows. 
If $\gamma \leq \B_{\cmp\alpha}(P_{\cmp\alpha}(\beta))$
then
$P_{\cmp\alpha}(\gamma)\leq P_{\cmp\alpha}(\beta)$ and, since a weak
comprehension $\cmp{\alpha\Land\gamma}:Z\to A$ of 
$\alpha\Land\gamma$ factors through $\cmp\alpha:X\to A$, then 
$$
\tt_Z\leq P_{\cmp{\alpha\Land\gamma}}(\alpha\Land\gamma)\leq
P_{\cmp{\alpha\Land\gamma}}(\gamma)\leq P_{\cmp{\alpha\Land\gamma}}(\beta)
$$
and, finally full comprehension yields that
$\alpha\Land\gamma\leq\beta$. Next, 
if $\alpha\Land\gamma\leq\beta$, then
$$
P_{\cmp\alpha}(\gamma)\leq\tt_X\Land P_{\cmp\alpha}(\gamma)\leq
P_{\cmp\alpha}(\alpha)\Land P_{\cmp\alpha}(\gamma)\leq
P_{\cmp\alpha}(\alpha\Land\gamma)\leq
P_{\cmp\alpha}(\beta).\eqno{\qedhere}
$$
\end{proof}

Consider the 2-full 2-subcategory \CH of \EH whose objects are
elementary doctrines with full
comprehensions
and  comprehensive equalizers.
\begin{roll}
\item[The 1-arrows] are those pairs $(F,b)$ in \EH
such that $F$ preserves comprehensions.
\end{roll}

\begin{remark}
The functor $F$ in a pair $(F,b)$ in \CH preserves all finite limits.
\end{remark}

The other aspect that we shall consider about set-like doctrines
is that every quotient should be of effective descent. We recall the
notion of descent data for a $P$-equivalence relation:

\begin{definition}
Given an elementary doctrine
$P:\ct{C}\op\longrightarrow\Ct{InfSL}$ and 
a $P$-e\-quiv\-a\-lence relation $\rho$ on an object $A$ in \ct{C},
the partial order of descent data \des{\rho} is the sub-order of
$P(A)$ on those $\alpha$ such that
$$P_{\pr_1}(\alpha)\Land_{A\times A}\rho\leq P_{\pr_2}(\alpha),$$
where $\pr_1,\pr_2:A\times A\to A$ are the projections.
\end{definition}

\begin{remark}
Given an elementary doctrine
$P:\ct{C}\op\longrightarrow\Ct{InfSL}$,
for $f:A\to B$ in \ct{C}, let $\rho$ be the $P$-kernel 
${P_{f\times f}(\delta_B)}$.
The functor $P_f:P(B)\to P(A)$ takes values in 
$\des\rho\subseteq P(A)$.
\end{remark}

\begin{definition}\label{efdes}
Given an elementary doctrine
$P:\ct{C}\op\longrightarrow\Ct{InfSL}$ and 
an arrow $f:A\to B$ in \ct{C}, let $\rho$ be the $P$-kernel 
${P_{f\times f}(\delta_B)}$. 
The arrow $f$ is \dfn{of effective descent} if the functor
$P_f:P(B)\to\des\rho$ is an isomorphism.
\end{definition}

\begin{exm}
In the example of the doctrine $S:\ct{S}\op\longrightarrow\Ct{InfSL}$
on the category of sets and functions, as in \ref{monoe}, every
canonical surjection $f:A\to A/\sim$, in the quotient of an
equivalence relation $\sim$ on $A$,
is of effective descent. The condition in \ref{efdes} recognizes the
fact that the subsets of the $A/\sim$ are in bijection with those
subsets $U$ of $A$ that are \dfn{closed with respect to} the
equivalence relation, in the sense that, for $a_1,a_2\in A$ such that
$a_1\sim a_2$ and $a_2\in U$, one has also that $a_1\in U$.
\end{exm}

Consider the 2-full 2-subcategory \QH of \CH whose objects are
elementary doctrines $P:\ct{C}\op\longrightarrow\Ct{InfSL}$ in  \CH
with stable effective quotients of $P$-equivalence
relations and of effective descent.
\begin{roll}
\item[The 1-arrows] are those pairs $(F,b)$ in \EH
such that $F$ preserves quotients and comprehensions.
\end{roll}

\begin{prop}\label{regq}
If $P:\ct{C}\op\longrightarrow\Ct{InfSL}$ is an elementary
doctrine in \QH, then the category \ct{C} is regular. Moreover,
if  $P:\ct{C}\op\longrightarrow\Ct{InfSL}$ is also existential,
every equivalence relation in \ct{C} has a stable coequalizer.
\end{prop}

\begin{proof}
The category \ct{C} has pullbacks and these can be computed by means
of comprehensions: in the case of interest, given an arrow $f:A\to B$,
the comprehension 
$\cmp{\rho}\equiv\cmp{P_{f\times f}(\delta_B)}:K\to A\times A$
of the $P$-kernel of $f$ gives a pullback
$$\xymatrix@=4em{
K\ar[r]^{\pr_1\circ\cmp\rho}\ar[d]_{\pr_2\circ\cmp\rho}&A\ar[d]^f\\
A\ar[r]_f&B}$$
in \ct{C} by \ref{dswc}. Moreover, thanks to fullness of the comprehension of $\rho$,
the quotient $q:A\to C$ of $\rho$ provides a
coequalizer 
$$\xymatrix@=4em{K\ar@<.5ex>[r]^{\pr_1\circ\cmp\rho}
\ar@<-.5ex>[r]_{\pr_2\circ\cmp\rho}&A\ar[r]^q&C}$$
which is stable thanks to the stability of quotients and
comprehensions in $P$.

\noindent
For the second part, suppose that the doctrine $P$ is existential.
An equivalence relation 
$\xymatrix@1{r:R\ \ar@{>->}[r]&A\times A}$ in \ct{C}
determines the $P$-equivalence relation $\D_{r}(\tt_R)$ on $A$. To
conclude note
that the mediating arrow $R\to\cmp{\D_{r}(\tt_R)}$ is
monic.\end{proof}

The following result is a direct consequence of those accomplished in
\cite{HughesJ:facsft}, see also \cite{PavlovicD:mapsii}.

\begin{prop}\label{exq}
Let $P:\ct{C}\op\longrightarrow\Ct{InfSL}$ be a primary doctrine.
$P$ is an existential elementary
doctrine in \QH where every
monomorphism in \ct{C} is a comprehension if
and only if \ct{C} is exact and $P$ is equivalent to
the doctrine $S:\ct{C}\op\longrightarrow\Ct{InfSL}$ of
subobjects.
\end{prop}

\section{Completing with quotients as a free construction}

There is a fairly obvious construction that produces an elementary
doctrine with quotients. We shall present it in the following and
prove that it satisfies a universal property.

\start{Let
$P:\ct{C}\op\longrightarrow\Ct{InfSL}$ be an elementary 
doctrine} for the rest of the section.
Consider the category $\ct{Q}_P$ 
of ``quotients in $P$'', 
the \dfn{elementary quotient completion of $P$}, defined as follows:
\begin{roll}
\item[an object of $\ct{Q}_P$] is a pair $(A,\rho)$ such that $\rho$
is a $P$-equivalence relation on $A$
\item[an arrow {$\ec{f}:(A,\rho)\to(B,\sigma)$}] is an equivalence class
of arrows $f:A\to B$ in \ct{C} (with a chosen representative) such that 
$\rho\leq_{A\times A}P_{f\times f}(\sigma)$ in $P(A\times A)$ with
respect to the relation determined by the condition that 
$\rho\leq_{A\times A}P_{f\times g}(\sigma)$
\end{roll}
Composition is given by that of \ct{C} on representatives, and
identities are represented by identities of \ct{C}.

The indexed partial inf-semilattice
$\Q{P}:\ct{Q}_P\op\longrightarrow\Ct{InfSL}$
on $\ct{Q}_P$ will be given by categories of descent data: 
on an object $(A,\rho)$ it is defined as
$$
\Q{P}(A,\rho)\colon=\des{\rho}
$$
and the following lemma is instrumental to give the assignment on
arrows by using the action of $P$ on (any)
representatives---in the sense that the action of \Q{P} on
arrows will then be defined as
$\Q{P}_{\ec{f}}\colon=P_f$ for $\ec{f}:(A,\rho)\to(B,\sigma)$.

\begin{lemma}\thmitem
With the notation used above, let $(A,\rho)$ and $(B,\sigma)$ be
objects in $\ct{Q}_P$, and let $\beta$ be an object in
\des{\sigma}.
\begin{enumerate}
\item If $f:A\to B$ is an arrow in \ct{C} such that
$\rho\leq_{A\times A}P_{f\times f}(\sigma)$, then $P_f(\beta)$
is in \des{\rho}. 
\item If $f,g:A\to B$ are arrows in \ct{C} such that 
$\rho\leq_{A\times A}P_{f\times g}(\sigma)$, then 
$$P_f(\beta)\gel P_g(\beta).$$
\end{enumerate}
\end{lemma}

\begin{proof}
(i) is immediate.\newline
(ii) Since $\beta$ is in \des{\sigma}, one has that
$$P_{\pr_1'}(\beta)\Land\sigma\leq_{B\times B}P_{\pr_2'}(\beta)$$
where $\pr_1',\pr_2':B\times B\to B$ are the two projections.
Hence
$$P_{f\times g}(P_{\pr_1'}(\beta))\Land
P_{f\times g}(\sigma)\leq_{A\times A}P_{f\times g}(P_{\pr_2'}(\beta))$$
since $P_{f\times g}$ preserves the structure. By the
hypothesis that $\rho\leq_{A\times A}P_{f\times g}(\sigma)$,
$$P_{f\circ\pr_1}(\beta)\Land\rho\leq_{A\times A}P_{g\circ\pr_2}(\beta)$$
where $\pr_1,\pr_2:A\times A\to A$ are the two projections. Taking
$P_{\Delta_A}$ of both sides and recalling reflexivity of $\rho$
$$P_f(\beta)\gel P_f(\beta)\Land\tt_A\gel
P_{\Delta_A}(P_{f\circ\pr_1}(\beta))\Land P_{\Delta_A}(\rho)\leq
P_{\Delta_A}(P_{g\circ\pr_2}(\beta))\gel P_g(\beta).$$
The other direction follows by symmetry.\end{proof}

\begin{lemma}
With the notation used above,
$\Q{P}:\ct{Q}_P\op\longrightarrow\Ct{InfSL}$ is a primary
doctrine.
\end{lemma}

\begin{proof}
For $(A,\rho)$ and $(B,\sigma)$ in $\ct{Q}_P$
let 
$\pr_1,\pr_3:A\times B\times A\times B\to A$ and
$\pr_2,\pr_4:A\times B\times A\times B\to B$ be the four
projections. The meet of two $P$-equivalence relations on 
$A\times B$
$$\rho\boxtimes\sigma\colon=P_{<\pr_1,\pr_3>}(\rho)
\Land_{A\times B\times A\times B}
P_{<\pr_2,\pr_4>}(\sigma)$$
is a $P$-equivalence relation on $A\times B$ and it
provides an object $(A\times B,\rho\boxtimes\sigma)$ which, together
with the arrows determined by the two projections from $A\times B$,
gives a product of $(A,\rho)$ and $(B,\sigma)$ in $\ct{Q}_P$.\newline
For each $(A,\rho)$, the sub-partial order $\des\rho\subseteq P(A)$ is
closed under finite meets.\end{proof}

\start{Assume that $P$ has weak comprehensions} for the rest of
the section.

\begin{lemma}\label{wscl}
With the notation used above, \Q{P} is an elementary doctrine
with comprehensions and comprehensive equalizers. 
If $P$ has full weak comprehensions,
then \Q{P} has full comprehensions.
\end{lemma}

\begin{proof}
First we show that $\ct{Q}_P$ has equalizers, hence all finite
limits. Consider a parallel pair
$\ec{f},\ec{g}:(A,\rho)\to(B,\sigma)$, 
and let
$e\colon=\cmp{P_{<f,g>}(\sigma)}:E\to A$ be a
weak comprehension. It is easy to see that
$\ec{e}:(E,P_{e\times e}(\rho))\to(A,\rho)$ is
an equalizer as required.\newline
A similar argument shows that $\Q{P}$ has comprehensions. More
precisely, given $(A,\rho)$ and $\beta$ in $\des\rho$, let
$\cmp\beta:X\to A$ be a weak comprehension for $\beta$ over $A$ in the 
doctrine $P$. A comprehension for $\beta$ over $(A,\rho)$ in
\Q{P} is
$$\ec{\cmp\beta}:(X,P_{\cmp\beta\times\cmp\beta}(\rho))\to(A,\rho).$$
Fullness of weak comprehensions in $P$ implies fullness of comprehensions
in \Q{P} because objects of \Q{P} on $(A,\rho)$ are descent data
related to $P$.\newline
A left adjoint $\QD_{\ec{\Delta_A}}$ for $\Q{P}_{\ec{\Delta_A}}$ is
computed by 
$$\QD_{\ec{\Delta_A}}(\alpha)\colon=
P_{\pr_1}(\alpha)\Land_{A\times A}\rho$$
for $\alpha$ in \des\rho.
Indeed, let $\theta$ be in \des{\rho\boxtimes\rho} such that
$\alpha\leq_{(A,\rho)}\Q{P}_{\ec{\Delta_A}}(\theta)$, 
{\it i.e.}  
$\alpha\leq_AP_{\Delta_A}(\theta)$. Thus
$\D_{\Delta_A}(\alpha)\leq_{A\times A}\theta$ and one has by remark~\ref{equ}
$$\begin{array}{r@{}l}
P_{\pr_1'}(\alpha)\Land P_{<\pr_1',\pr_2'>}(\delta_A)\Land
P_{<\pr_2',\pr_3'>}(\rho)&{}\leq_{A\times A\times A}
P_{<\pr_1',\pr_2'>}(\theta)\Land P_{<\pr_2',\pr_3'>}(\rho)\\[1.5ex]
&{}\leq_{A\times A\times A}P_{<\pr_1',\pr_3'>}(\theta)
\end{array}$$
for $\pr_i':A\times A\times A\to A,\ i=1,2,3$, the three
projections.
Hence $P_{\pr_1}(\alpha)\Land\rho\leq_{A\times A}\theta$, which gives
$\QD_{\ec{\Delta_A}}(\alpha)
\leq_{(A\times A,\rho\boxtimes\rho)}\theta$. It is easy 
to prove the converse that, if
$\QD_{\ec{\Delta_A}}(\alpha)\leq\theta$, then
$\alpha\leq\Q{P}_{\ec{\Delta_A}}(\theta)$.\newline
More generally, a left adjoint $\QD_{\ec{id_C\times \Delta_A}}$ for $\Q{P}_{\ec{id_C\times\Delta_A}}$
is computed analogously  by 
$$\QD_{\ec{id_C\times\Delta_A}}(\alpha)\colon=
P_{<\pr_1,\pr_2>}(\alpha)\Land_{C\times A\times A}P_{<\pr_2,\pr_3>}(\rho)$$
for $\alpha$ in $\Q{P}((C,\sigma)\times (A,\rho))$ with 
$\pr_1:C\times A\times A\rightarrow C$ and 
$\pr_i: C \times A\times A\rightarrow A$ for $i=2,3$ 
the corresponding projections.\newline
Finally, $\Q{P}$ has comprehensive equalizers by construction. 
\end{proof}

A careful reader may have noticed that the hypothesis on weak
comprehensions in $P$ in \ref{wscl}
was needed in order to construct all pullbacks
and characterize regular monos in $\ct{Q}_P$. 

\begin{lemma}
With the notation used above, 
\Q{P} has effective quotients of \Q{P}-equivalence 
relations and those are of effective descent.
\end{lemma}

\begin{proof}
Since the sub-partial order $\des\rho\subseteq P(A)$ is
closed under finite meets, a \Q{P}-equivalence relation $\tau$ on
$(A,\rho)$ is also a $P$-equivalence relation on $A$. It is easy to see
that $\ec{\id{A}}:(A,\rho)\to(A,\tau)$ is an effective quotient and of
effective descent since $\rho\leq_{A\times A}\tau$.\end{proof}

\begin{lemma}
With the notation used above, 
a quotient of \Q{P}-equivalence is stable.
\end{lemma}

\begin{proof}
Let $\tau$ be a \Q{P}-equivalence relation on
$(A,\rho)$, let $\ec{\id{A}}:(A,\rho)\to(A,\tau)$ be its quotient, and
let $\ec{f}:(B,\sigma)\to(A,\tau)$ be an arrow in $\ct{Q}_P$.
By the previous lemma \ref{wscl}, in $\ct{Q}_P$ there is a pullback
diagram
$$\xymatrix@C=1.5ex@R=2.5em{(E,\upsilon)\ar[rr]\ar[dd]
\ar[rd]|{\EC{\cmp{P_{f\times\id{A}}(\tau)}}}
&&(B,\sigma)\ar[dd]^{\EC{f}}\\
&(B\times A,\sigma\boxtimes\rho)
\ar[ru]_(.6){\EC{\pr_1}}\ar[ld]^(.4){\EC{\pr_2}}\\
(A,\rho)\ar[rr]_{\EC{\id{A}}}&&(A,\tau)}$$
where 
$\upsilon\colon=
P_{\cmp{P_{f\times\id{A}}(\tau)}\times\cmp{P_{f\times\id{A}}(\tau)}}
(\sigma\boxtimes\rho)$.
Let $e\colon=\cmp{P_{f\times\id{A}}(\tau)}$ 
and $\omega\colon=P_{e\times e}(\sigma\boxtimes\tau)$.
The arrow $<\id{B},f>:B\to B\times A$ of $\ct{C}$ factors through the
comprehension $\cmp{P_{f\times\id{A}}(\tau)}:E\to B\times A$ because
$$\tt_B\leq P_{f}(P_{\Delta_A}(\tau))=
P_{<f,f>}(\tau)=P_{<\id{B},f>}(P_{f\times\id{A}}(\tau)),$$
say $<\id{B},f>=\cmp{P_{f\times\id{A}}(\tau)}\circ g= e \circ g$. Moreover
$$\sigma\leq P_{<\id{B},f>\times<\id{B},f>}(\sigma\boxtimes\tau)=
P_{g\times g}(P_{e\times e}(\sigma\boxtimes\tau))=
P_{g\times g}(\omega)$$
hence producing a commutative diagram
$$\xymatrix@C=.5ex@R=3em{
(E,\omega)\ar[rd]_(.4){\EC{\cmp{P_{f\times\id{A}}(\tau)}}}&&
(B,\sigma)\ar[ll]^{\EC{g}}\ar[dl]^(.4){\EC{<\id{B},f>}}\\
&(B\times A,\sigma\boxtimes\tau)}$$
in $\ct{Q}_P$.
It is easy to see that $\ec{g}:(B,\sigma)\to(E,\omega)$ has a retraction 
$$\ec{\pr_1\circ\cmp{P_{f\times\id{A}}(\tau)}}:(E,\omega)\to(B,\sigma)$$
which is monic and hence an inverse of $\ec{g}$.
Therefore $(B,\sigma)$ becomes a quotient of $\omega$ on
$(E,\upsilon)$.\end{proof}

There is a 1-arrow $(J,j):P\to\Q{P}$ in \EH. The functor
$J:\ct{C}\to\ct{Q}_P$ sends an object $A$ in \ct{C} to
$(A,\delta_A)$ and an arrow $f:A\to B$ to 
$\ec{f}:(A,\delta_A)\to(B,\delta_B)$ since 
$\delta_A\leq_{A\times A}P_{f\times f}(\delta_B)$ by \ref{exer}(a) and
\ref{reme}. The functor $J$ is full. 

For $A$ in \ct{C}, the partial order
$\Q{P}(A,\delta_A)=\des{\delta_A}$ is  $P(A)$ since 
$P_{\pr_1}\alpha\Land\delta_A\leq_{A\times A}P_{\pr_2}\alpha$ for any 
$\alpha$ in $P(A)$. Take the function $j_A:P(A)\to\Q{P}(A,\delta_A)$
to be the identity.

\begin{lemma}\label{cov}\thmitem
With the notation used above,
\begin{enumerate}
\item every object in $\ct{Q}_P$ is a quotient of a \Q{P}-equivalence
relation on an object in the image of $J$
\item every object in the image of $J$ is projective with respect to
quotients of \Q{P}-e\-quiv\-a\-lence relation.
\end{enumerate}
\end{lemma}

\begin{proof}
(i) The object $(A,\rho)$ is a quotient of $(A,\delta_A)$.\newline
(ii) A quotient of a \Q{P}-equivalence relation is isomorphic to one
  that has an identity as a representative.\end{proof}

\start{Assume that $P$ has comprehensive weak equalizers} for the
rest of the section.

\begin{lemma}\label{wec}
With the notation used above,  the functor
$J:\ct{C}\to\ct{Q}_P$ is faithful. 
\end{lemma}

\begin{proof}
For $f,g:A\to B$ suppose that $J(f)=J(g)$. In other words,
$$\D_{\Delta_A}(\tt_A)=\delta_A\leq_{A\times A}P_{f\times g}(\delta_B)$$
and equivalently
$\tt_A\leq_AP_{\Delta_A}(P_{f\times g}(\delta_B))
\gel_AP_{<f, g>}(\delta_B)$
since $\D_{\Delta_A}\dashv P_{\Delta_A}$. It follows that the identity
on $A$ equalizes $f$ and $g$, thanks to lemma~\ref{dswc} and the
hypothesis that $P$ has comprehensive weak equalizers.\end{proof}

The doctrine $P:\ct{C}\op\longrightarrow\Ct{InfSL}$ in
\ref{wec} is obtained from
$\Q{P}:\ct{Q}_P\op\longrightarrow\Ct{InfSL}$
by change of base along $J$ since $j$ is the identity natural
transformation. In logical terms, this states that the theory 
expressing $P$ extended with a quotient constructor is conservative
over the original theory.

We can now prove that the assignment $P\mapsto\Q{P}$ gives a left
bi-adjoint to the forgetful 2-functor $U:\QH\to\CH$.

\begin{theorem}\label{mthm}
For every elementary doctrine
$P:\ct{C}\op\longrightarrow\Ct{InfSL}$ in \CH
pre-composition with the 1-arrow
$$
\xymatrix@C=4em@R=1em{
{\ct{C}\op}\ar[rd]^(.4){P}_(.4){}="P"\ar[dd]_J&\\
           & {\Ct{InfSL}}\\
{\ct{Q}_P\op}\ar[ru]_(.4){\Q{P}}^(.4){}="R"&\ar"P";"R"_j^{\kern-.4ex\cdot}}
$$
in \CH induces an essential equivalence of categories 
\begin{equation}\label{eqv}
-\circ(J,j):\QH(\Q{P},X)\equiv\CH(P,X)
\end{equation}
for every $X$ in \QH.
\end{theorem}

\begin{proof}
Suppose $X$ is a doctrine in \QH.
As to full faithfulness of the functor in (\ref{eqv}), consider two
pairs $(F,b)$ and $(G,c)$ of 1-arrows from \Q{P} to $X$. By 
\ref{cov}, the natural transformation $\theta:F\stackrel.\to G$ in
a 2-arrow from $(F,b)$ to $(G,c)$ in \QH is completely 
determined by its action on objects in the image of $J$.
And, since a
quotient $q:U\to V$ of an $X$-equivalence relation on $U$
is of effective descent, $X(V)$ is a full sub-partial order of $X(U)$.
Essential surjectivity of the functor in (\ref{eqv}) follows from
\ref{cov}.
Preservation of comprehensions for $J$ follows from \ref{dm}
because a comprehension is monic.\end{proof}

The reason why $-\circ(J,j)$ need not be a strong equivalence in
\ref{mthm} is that quotients in the doctrine $X$ are determined
only up to iso. Since quotients in the doctrine \Q{P} are
determined by a specified construction, it is possible to strengthen
the result by restricting the 2-categories \QH and \CH to doctrines
with a choice of the categorical structure.

\begin{exms}\label{eqc}

\noindent(a)
The structure of the elementary quotient completion for the doctrine 
$LT:\ct{V}\op\longrightarrow\Ct{InfSL}$ from a first order theory
$\mathcal{T}$ 
with an equality predicate depends
heavily on the choice of basic operation symbols used to present the
theory $\mathcal{T}$, see \ref{wfs} and \ref{ccqc}.

\noindent(b)
For a cartesian category \ct{C} with weak pullbacks, 
the elementary quotient completion of the elementary doctrine 
$\Psi:\ct{C}\op\longrightarrow\Ct{InfSL}$ given in \ref{weaklhy}
is essentially equivalent to the exact completion of \ct{C}, see
\cite{CarboniA:regec}.
\end{exms}

\begin{remark}\label{exlexcom}
By \ref{regq}, the elementary quotient completion $\ct{Q}_S$ 
of an elementary doctrine $S:\ct{X}\op\longrightarrow\Ct{InfSL}$ in
\CH, obtained from a category \ct{X} with 
products and pullbacks as in \ref{monoe}, is regular.
Then, it may seem natural to compare $\ct{Q}_S$ 
with the regular completion of \ct{X}, see \cite{CarboniA:somfcr}, or
with the exact/regular completion of \ct{X} in case \ct{X} is
regular, see \cite{FreydP:cata,CarboniA:somfcr}. The three
constructions are in general not equivalent; that will be treated in
future work. Also, an example of an elementary quotient completion
that is not exact will be presented in section~\ref{twl}.
\end{remark}

In case an elementary doctrine
$P:\ct{C}\op\longrightarrow\Ct{InfSL}$ with comprehensive equalizers
fails to have \emph{all} comprehensions, it is possible 
to add these freely (hence fully) preserving the rest of the structure
by taking the fibration of vertical maps on the total category of $P$,
see \cite{jacobbook,TaylorP:prafom}.

\section{Applications to richer doctrines}\label{appl}

We shall now analyse the extent to which the elementary quotient
completion construction produced in the previous section behaves with respect to
further logical structure. We start with existential doctrines.

\begin{theorem}\label{exi}
Suppose $P:\ct{C}\op\longrightarrow\Ct{InfSL}$ is an existential
elementary doctrine. Then the
doctrine $\Q{P}:\ct{Q}_P\op\longrightarrow\Ct{InfSL}$ is existential
and the pair $(J,j):P\to\Q{P}$ is a 1-arrow in \EE.
\end{theorem}

\begin{proof}
The left adjoint along a projection with respect to $\Q{P}$ is given
by the left adjoint along the underlying projection in $\ct{C}$ with
respect to $P$.\end{proof} 

We now consider other logical structures such as implications and
universal quantification, always remaining within Lawvere's unifying
schema of ``logical constructs as adjoints'', see 
\cite{LawvereF:adjif,jacobbook}.

\begin{definition}\label{imh}
A primary doctrine 
$P:\ct{C}\op\longrightarrow\Ct{InfSL}$
is \dfn{implicational} if, for every object $A$ in
\ct{C}, every $\alpha$ in $P(A)$, the functor
$\alpha\Land\blank:P(A)\to P(A)$  
has a right adjoint $\alpha\Implies\blank:P(A)\to P(A)$.
\end{definition}

\begin{exms}
\noindent(a)
The primary fibration $LT:\ct{V}\op\longrightarrow\Ct{InfSL}$, as
defined in \ref{lta} for a first order theory $\mathcal{T}$, is
implicational since, for $V$ and $W$ in  
$LT(\vec x)$, the formula $V\Implies W$ gives the value of the right
adjoint of $V\Land\blank$ on $W$.

\noindent(b)
The primary doctrine in example~\ref{model} is implicational: on a
subset $P$ of $A$, the right adjoint $P\Implies\blank$ on a subset
$Q$ of $A$ is computed as
$$P\Implies Q\colon=\left\{a\in A\mid 
a\in P\Implies a\in Q\right\}$$
\end{exms}

\begin{definition}\label{unh}
A primary doctrine 
$P:\ct{C}\op\longrightarrow\Ct{InfSL}$
is \dfn{universal}
if, for $A_1$ and $A_2$ in \ct{C}, for a(ny) projection
$\pr_i:A_1\times A_2\to A_i$, $i=1,2$,
the functor $P_{\pr_i}:P(A_i)\to P(A_1\times A_2)$ has a right adjoint
$\B_{\pr_i}$, and these satisfy the Beck-Chevalley condition:\newline
for any pullback diagram
$$\xymatrix{X'\ar[r]^{\pr'}\ar[d]_{f'}&A'\ar[d]^f\\X\ar[r]^{\pr}&A}$$
with $\pr$ a projection (hence also $\pr'$ a projection), for any
$\beta$ in $P(X)$, the canonical arrow 
$ P_f\B_\pr(\beta)\leq \B_{\pr'}P_{f'}(\beta)$ in $P(A')$ is iso; 
\end{definition}

\begin{exms}
\noindent(a)
The primary fibration $LT:\ct{V}\op\longrightarrow\Ct{InfSL}$, as
defined in \ref{lta} for a first order theory $\mathcal{T}$, is
universal. A  right adjoint to $P_\pr$ is computed by quantifying
universally the variables that are not involved in the substitution
given by the projection, {\it e.g.} for the projection
$\pr=[z/x]:(x,y)\to(z)$ and a formula $W$ with free variables at most
$x$ and $y$, $\B_\pr(W)$ is $\Forall y.(W[z/x])$.
As in the case of the left adjoints in \ref{exh}, the Beck-Chevalley
condition expresses the correct interplay between term substitution
and universal quantification.

\noindent(b)
The primary doctrine in example~\ref{model} is universal: on a subset
$U$ of $A$, the adjoint $\B_\pr$, for a projection $\pr$, is given by
$\B_\pr(U)\colon=\left\{b\in B\mid
\pr^{-1}\left\{b\right\}\subseteq U\right\}.$
\end{exms}

\begin{remark}\label{hure}
In an elementary doctrine $P:\ct{C}\op\longrightarrow\Ct{InfSL}$
which is implicational and universal
for every map $f:A\to B$ in \ct{C} the functor $P_f$ has a right
adjoint $\B_f$ that can be
computed as 
$$\B_{\pr_2}(P_{f\times \id B}(\delta_B)\Implies P_{\pr_1}(\alpha))$$
for $\alpha$ in $P(A)$, where $\pr_1$ and $\pr_2$ are the projections
from $A\times B$---similarly to \ref{here}.

Moreover, if $P$ has full weak comprehensions, then the converse also
holds by~\ref{wfs}.
\end{remark}

The following is an instance of a general situation for the quotient
completion: a weak representation in the doctrine $P$ gives
rise to a strict representation in the completion \Q{P}. Recall
from \cite{RosoliniG:loccce} that a 
\dfn{weak evaluation from $A$ to $B$} is an arrow 
$w:\W{A}{B}\times A\to B$
such that for every arrow $f:X\times A\to B$ there is an arrow
$f':X\to\W{A}{B}$ with $f=w\circ(f'\times\id{A})$. In
{\it loc.cit.}, the notion above was instrumental to give a
characterization of those exact completions which are locally
cartesian closed; it helps to show how the more general structure in
the present paper extends that of exact completion, see also
\cite{OostenJ:reaait}. We follow \cite{RosoliniG:loccce} and say that
\ct{C} is \dfn{weakly cartesian closed} if it has a weak evaluation
for each pair of objects.

\begin{prop}\label{ccqc}
Let $P:\ct{C}\op\longrightarrow\Ct{InfSL}$ be an elementary,
implicational, universal doctrine with weak comprehensions. If \ct{C}
is weakly cartesian closed, then $\ct{Q}_P$ is cartesian closed and
$\Q{P}$ is implicational and universal. Moreover, the 1-arrow 
$(J,j):P\to\Q{P}$ 
preserves the implicational and universal structure.
\end{prop}

\begin{proof}
Suppose that $(A,\rho)$ and $(B,\sigma)$ are objects in $\ct{Q}_P$
and $w:\W{A}{B}\times A\to B$ is a weak evaluation from $A$ to $B$ in
\ct{C}. Recalling \ref{wfs}, consider the object $\varphi$ of
$P(\W{A}{B}\times\W{A}{B})$:
$$
\varphi\, =\, \B_{\pr_3}\B_{\pr_4}(P_{<\pr_3,\pr_4>}(\rho)\Implies
P_{<w<\pr_1,\pr_3>,w<\pr_2,\pr_4>>}(\sigma))
$$
where $\pr_1,\pr_2:\W{A}{B}\times\W{A}{B}\times A\times A\to\W{A}{B}$
and $\pr_3,\pr_4:\W{A}{B}\times\W{A}{B}\times A\times A\to A$ are the
projections. 
Within the logic provided by the doctrine, it can be described as
the intension of those pairs $(e_1,e_2)$ of ``elements of $\W{A}{B}$''
such that
$$
\Forall a_3\in A.\Forall a_4\in A.(a_3\mathrel{\rho}a_4\Implies
w(e_1,a_3)\mathrel{\sigma}w(e_2,a_4))
$$
where we have used the same indices for the variables as those for the
projections. It is easy to see that $\varphi$ satisfies the conditions of
symmetry and transitivity in \ref{per} so that, considered
a weak comprehension $\cmp{P_{\Delta_{\W{A}{B}}}(\varphi)}:F\to\W{A}{B}$ of
$P_{\Delta_{\W{A}{B}}}(\varphi)$ in $P(\W{A}{B})$, the object
$\psi\colon=
P_{\cmp{P_{\Delta_{\W{A}{B}}}(\varphi)}\times\cmp{P_{\Delta_{\W{A}{B}}}(\varphi)}}
(\varphi)$
is a $P$-equivalence relation. It is straightforward to prove that
$$w\circ(\cmp{P_{\Delta_{\W{A}{B}}}(\varphi)}\times\id{A}):
(F,\psi)\times(A,\rho)\to(B,\sigma)$$
is a (strict) evalutation from
$(A,\rho)$ to $(B,\sigma)$.

The right adjoints required for the implicational and universal
structures on $\Q{P}$ are computed by
the corresponding right adjoints with respect to $P$.
\end{proof}

As an application, we shall consider the Axiom of Unique Choice
(AUC) within an existential
doctrine $P:\ct{C}\op\longrightarrow\Ct{InfSL}$ with the same
structure as in \ref{ccqc}. Recall that (AUC) states that, in a
cartesian closed regular category, a total
relation, which is {\em also} single-valued, contains the graph of a
function, see {\it e.g.} \cite{streicher}. We shall adapt it to the
more general frame of 
existential doctrines as in \ref{ccqc}.
Let $w:\W{A}{B}\times A\to B$ be a weak evaluation from $A$ to $B$ in
\ct{C}. 
We say that (AUC) \dfn{from $A$ to $B$ holds for $w$} if,
for every $\rho$ in $P(A\times B)$,
$$
\def\mp#1_#2{{\displaystyle\mathop{#1}_{\scriptstyle\kern1ex #2}}}
\begin{array}{@{}l@{}l@{}r@{}}
\mp\B_{A\to1}\mp\D_{A\times B\to A}(\rho)\Land\\
\multicolumn3{l}
{\ \Land\mp\B_{B\times B\to1}\mp\B_{A\times B\times B\to B\times B}
((P_{<\pr_1,\pr_2>}(\rho)\Land P_{<\pr_1,\pr_3>}(\rho))\Implies
P_{<\pr_2,\pr_3>}(\delta_B))\leq}\\
&&\leq\mp\D_{\W{A}{B}\to1}\mp\B_{\W{A}{B}\times A\to\W{A}{B}} 
P_{<\pr_2',w >}(\rho)
\end{array}
$$
where we dispensed with the labels on the various appropriate
projection arrows as these are easily reconstructed.

Using plain logical notation (as in the proof of \ref{ccqc}), 
$$\begin{array}{@{}l@{}l@{}r@{}}
\Forall x_1\in A.\Exists y_2\in B.\ x_1\mathrel{\rho}y_2\Land\\
\multicolumn3{l}
{\Land\Forall y_2\in B.\Forall y_3\in B.\Forall x_1\in A. 
((x_1\mathrel{\rho}y_2\Land x_1\mathrel{\rho}y_3)\Implies
y_2\mathrel{\delta_B} y_3\leq}\\
&&\leq\Exists h_1\in\W{A}{B}.\Forall x_2\in A.\
x_2\mathrel{\rho}w(h_1,x_2).
\end{array}$$

We skip the easy proof that, if (AUC) from $A$ to $B$ holds for a weak
evaluation, then it holds for every weak evaluation. For that reason,
from now on we say that (AUC) \dfn{from $A$ to $B$ holds}.

\begin{cor}\label{uac}\thmitem
Suppose that $P:\ct{C}\op\longrightarrow\Ct{InfSL}$ is an elementary 
existential, implicational and universal doctrine with weak
comprehensions. Suppose that \ct{C} is weakly cartesian closed and
that $A$ and $B$ are objects in \ct{C}. The
following are equivalent:
\begin{enumerate}
\item in $P$ {\normalfont(AUC)} from $A$ to $B$ holds
\item in $\Q{P}$ {\normalfont(AUC)} from $(A,\delta_A)$ to
$(B,\delta_B)$ holds. 
\end{enumerate}
\end{cor}

\begin{proof}
It follows directly from the preservation properties of the 1-arrow
$(J,j):P\to\Q{P}$ by proposition~\ref{exi} and \ref{ccqc}.\end{proof}

\section{Examples from constructive foundations}
The main applications of the elementary quotient completion appear in
the formalization of constructive mathematics in type theory. As we
already explained, the ability of constructing quotients of
equivalence relations is an essential tool of standard mathematics;
any formalization of constructive mathematics that intends to achieve
foundational relevance must allow treatment of such quotients in some
way, in particular if it is inherently intensional. We shall sketch
two examples of elementary quotient completion in type theory: one is
given over intensional Martin-L{\"o}f's type theory~\cite{PMTT} and
the other  over the intensional level of the minimalist
foundation \cite{m09}. Both models are based on total setoids \`a la
Bishop \cite{BISHOP}, with proof-terms as morphisms. 

\subsection{The setoid model over Martin-L{\"o}f's Type Theory}
The model of ``setoids'' over intensional Martin-L{\"o}f's Type 
Theory, see \cite{notepal,dyb}, is an instance  
of the exact completion of a category with finite products and weak
equalizers. We already know that it is an elementary quotient
completion, so we simply review the elementary doctrine one 
can construct in order to obtain it as such a completion.

The base category is a syntactic construction $\ct{ML}$
is defined as follows.
\begin{roll}
\item[The objects of $\ct{ML}$] are closed sets of MLTT.
\item[An arrow $\ec{t\in B\ [x\in A]}:A\to B$] is an equivalence
class of terms $t\in B\ [x\in A]$\footnote{To represent a term
depending on a variable $x$, we shall indifferently use
the metavariable $t$ or $t(x)$. We use the latter when we
need to refer to a substitution of a variable $x$ that may
appear in $t$.} derivable in MLTT where two
terms $t\in B\ [x\in A]$ and  $t'\in B\ [x\in A]$ are equivalent
if there is a proof-term in MLTT
$$p\in \mathsf{Id}(B, t, t')\ [x\in A]$$
\end{roll}
The composition of $\ec{t\in B\ [x\in A]}:A\to B$ and
$\ec{s(x)\in C\ [x\in B]}$ is defined as
$$\ec{s(t)\in C\ [x\in A]}:A\to C.$$
The identity morphism on $A$ is $\ec{x\in A\ [x\in A]}:A\to A$.

\begin{lemma}\label{synmltt}
The category $\ct{ML}$ has finite products and weak equalizers.
\end{lemma}

\begin{proof}
A terminal object is the singleton $N_1$.
A product of closed sets $A,B$ is $\ds_{x\in A}B$
with its first and second projections since
$$
\mathsf{idpeel}(z,(x,y).\mathsf{id}(<x,y>))\in
\mathsf{Id}(\ds_{x\in A}B,<\pi_1(z),\pi_2(z)>,z)$$
thanks to the $\beta$-rules.\newline
A weak equalizer of $\ec{t\in B\ [x\in A]},
\ec{t'\in B\ [x\in A]}:A\to B$ is 
$$\ec{\pi_1(z)\in A\ [z\in \ds_{x\in A} \mathsf{Id}(B, t,t')]}:
\ds_{x\in A} \mathsf{Id}(B, t,t')\to A\eqno\qedhere$$
\end{proof}

The functor $\F:\ct{ML}\op\longrightarrow\Ct{InfSL}$
is defined on a closed set $A$ as the partial order consisting of
\begin{roll}
\item[equivalence classes $\ec{\phi\prp [x\in A]}$ of
predicates] in MLTT depending on $A$ with respect to
equiprovability, {\it i.e.} 
$\phi\prp [x\in A]\sim\phi'\prp [x\in A]$ if
there is a proof of
$\phi\leftrightarrow\phi'\prp [x\in A]$ in MLTT,
\item
[$\ec{\phi\prp [x\in A]}\leq\ec{\psi\prp [x\in A]}$]
if there is a proof-term 
$r\in\psi\ [x\in A, w\in \phi]$ in MLTT.
\end{roll}
The action of the functor on arrows in \ct{ML} is given by
substitution.

\begin{prop}\label{mlttd}
The doctrine $\F:\ct{ML}\op\longrightarrow\Ct{InfSL}$
is existential elementary, with full weak
comprehensions and comprehensive weak equalizers.
\end{prop}

\begin{proof}
Products of propositions and the singleton set provide finite meets
thanks 
to the identification of propositions with sets typical of MLTT.
The left adjoint to substitution with
$\ec{t(x)\in B\ [x\in A]}:A\to B$ is computed, for 
$\phi(x)\prp[x\in A]$,
as the equivalence class represented by the proposition
$$
\ds_{x\in A}(\mathsf{Id}(B,t(x),y)\wedge\phi(x))\prp [y\in B].
$$
A weak comprehension $\cmp\phi$ of $\phi\prp [x\in A]$
is given by 
$$\ec{\pi_1(w)\in A\ [w\in \ds_{x\in A}\phi]}:
\ds_{x\in A}\phi\to A.$$
Note that weak comprehensions are full: it is enough to show
a proof of $\phi\leq\exists_{\cmp\phi}\top_W$ (with $W$ codomain
of $\cmp\phi$) given by
$$
<<y,p>,<\mathsf{id}(y),\ast>>\in 
\ds_{w\in W}\left(\mathsf{Id}(A,\pi_1(w),y)
\wedge N_1 \right) [y\in A,p\in \phi].
$$
with $\ast$ the canonical element of the singleton set  $N_1$.
It follows by an immediate, direct calculation that weak equalizers are
comprehensions.
\end{proof}

The doctrine $\F:\ct{ML}\op\longrightarrow\Ct{InfSL}$
enjoys more properties.

\begin{prop}
\label{mlbc}
The category \ct{ML} is weakly cartesian closed.
For every arrow $f$ in \ct{C}, the functor $\F_f$ has also a right
adjoint $($which necessarily satisfies Beck-Chevalley condition$)$. 
\end{prop}

\begin{proof}
The category \ct{ML} is weakly cartesian closed because given
sets $A$ and $B$, a weak evaluation is defined via the dependent
product as 
$$\ec{\pi_1(z)(\pi_2(z))\in B\ 
\left[z\in\left(\dpd_{x\in A} B\right)\times A\right]}:
\left(\dpd_{x\in A} B\right)\times A\to B.$$
It is only a weak evalutation because 
inhabitation of $\mathsf{Id}(B(x),f(x),g(x))\ [x\in A]$ cannot
ensure inhabitation of 
$\mathsf{Id}(\dpd_{x\in A}B(x),\lambda x.f(x),\lambda x. g(x))$, see \cite{disttheshof}.
A right adjoint to substitution is given by the dependent product set.
\end{proof}

In fact, the doctrine  $\F$ is
isomorphic to the doctrine 
$\Psi^{\ct{ML}}:\ct{ML}\op\longrightarrow\Ct{InfSL}$,
constructed as in \ref{weaklhy} from the category \ct{ML}. The
isomorphism 
$$(\Id{\ct{ML}},h):\F\to\Psi^{\ct{ML}}$$
has the identity functor in the first component and, for $A$ a closed
set, the functor $h(A):\F(A)\to\Psi^{\ct{ML}}(A)$ maps 
$\ec{\phi(x)\textit{ prop\/}\ [x\in A]}$ to the equivalence class
represented by the arrow
\begin{equation}\label{ir}
\ec{\pi_1(w)\in A\ [w\in \ds_{x\in A}\phi(x)]}:
\ds_{x\in A}\phi(x)\to A
\end{equation}
which has codomain $A$. 
Note that a proof-term
$t(y)\in \psi\ [x\in A, y\in \phi]$ produces a
map and a commutative triangle
$$\xymatrix@C=7em{{\ds_{x\in A}\phi}
\ar[rr]^{\ec{<\pi_1(w),t(\pi_2(w))>\in\ds_{x\in A}\psi\ 
 [w\in\ds_{x\in A}\in \phi]}}
\ar[rd]_{\ec{\pi_1(w)\in A\ [w\in \ds_{x\in A}\phi]}\kern5em}&&
{\ds_{x\in A}\psi}
\ar[ld]^{\kern5em\ec{\pi_1(w)\in A\ [w\in \ds_{x\in A}\psi]}}\\
& A}$$
which shows that the assignment for $h(A)$ given in (\ref{ir}) does
not depend on the choice of representatives and it extends to a
functor. Moreover, one can easily check that the functor 
$k(A):\Psi^{\ct{ML}}(A)\to\F(A)$ mapping $f(y)\in A\ [y\in B]$ to
$\ds_{y\in B}\mathsf{Id}(A,f(y),x)$ for $x\in A$ provides an inverse to
$h(A)$.

\begin{remark}\label{mlttequiv}
The base category $\ct{Q}_{\F}$ of the elementary quotient completion
$\Q{\F}$ of $\F:\ct{ML}\op\longrightarrow\Ct{InfSL}$ is essentially
equivalent to the model of total setoids over Martin-L\"of type theory
as given in \cite{notepal}.
\end{remark}

\subsection{The quotient completion in the two-level minimalist foundation}\label{twl}
One of the motivations for developing the  notion of elementary quotient
completion  is to describe abstractly the ``quotient
model'' given in \cite{m09}.
That model is built over a dependent type theory acting as the
intensional level of the two-level minimalist foundation for 
constructive mathematics  in \cite{m09}.  It is used 
to give an interpretation of the extensional level of the
foundation---denoted \emtt 
for {\it extensional minimal type theory}---in the intensional
level---denoted \mtt for {\it minimal type theory}. 
Quite similar to that built over Martin-L{\"o}f's type theory, also
this model is based on the idea of total setoids \`a la Bishop, but
it appears to have different properties from that. For instance, the
setoid model built over a set-theoretic fragment of \mtt is not
exact as we note in the following---hence it cannot be an exact
completion.

Recall from \cite{m09} that \mtt is a predicative 
version of Coquand's Calculus of Constructions \CTT, see
\cite{tc90}. Indeed, as \CTT, it resembles a {\it many sorted logic}
with propositions given primitively and depending on
sorts. Propositions are equipped with proof-terms and sorts include
dependent sets as in Martin-L{\"o}f's type theory. But, contrary to
\CTT, in \mtt there are two types of sorts on which  propositions
depend: these are sets and collections, where the first are included
in the latter; this distinction resembles the
usual distinction between sets and classes in NBG and is instrumental
to represent the power collection of a set which, from a 
predicative point of view, need not be a set. 
However the whole \mtt can be interpreted in \CTT
if one interprets both sets and collections in \mtt as sets in \CTT.
Indeed, if in \mtt collections were identified with sets
then one would just get \CTT.
Therefore, many properties of the quotient model over \mtt can be
extended to the corresponding quotient model over \CTT. 

Corresponding to sets and collections in \mtt we
consider two doctrines: one is a sub-doctrine of
the other. The larger has the base category
formed by collections with their typed terms, the other has the
base category restricted to the full subcategory on sets. 

The ``syntactic category of collections in \mtt'' $\CM$ is
constructed like the previous example. It has 
closed collections of \mtt as objects. An arrow 
$\ec{t\in B\ [x\in A]}:A\to B$ from the closed collection $A$ to
the closed collection $B$ is an equivalence class of
proof-terms in \mtt where 
$t\in A\ [x\in B]$ is equivalent to $t'\in A\ [x\in B]$
if there is a proof-term 
$$p\in \mathsf{Id}(A,t,t')\ [x\in B]$$
in \mtt.
The composition of the two arrows $\ec{t\in B\ [x\in A]}:A\to B$ and 
$\ec{s(y)\in C\ [y\in B]}:B\to C$
is given as the equivalence class
$$\ec{s(t)\in C\ [x\in A]}:A\to C$$
with identity arrows given by
$\ec{x\in A\ [x\in A]}:A\to A$.

The ``syntactic category of sets in \mtt'' \SM is the full
subcategory of \CM on the closed sets.

Like in the previous example, the two categories just introduced have
weak limits. 

\begin{lemma}\label{synmtt}
The categories $\CM$ and $\SM$ are cartesian and with weak
equalizers. 
\end{lemma}
\begin{proof}
The proof is analogous to that of \ref{synmltt}.
\end{proof}

We now define a doctrine on each syntactic category defined on
\mtt in a similar way to that for Martin-L{\"o}f's type
theory.

The functor
$$\Fm:\CM\op\longrightarrow\mathsf{InfSl}$$
is defined on a closed collection $A$ as the partial order 
consisting of 
\begin{roll}
\item[equivalence classes $\ec{\phi\prp [x\in A]}$ of
predicates] in \mtt depending on $A$ with respect to
equiprovability, {\it i.e.} 
$\phi\prp [x\in A]\sim\phi'\prp [x\in A]$ if
there is a proof of
$\phi\leftrightarrow\phi'\prp [x\in A]$ in \mtt,
\item
[$\ec{\phi\prp [x\in A]}\leq\ec{\psi\prp [x\in A]}$]
if there is a proof-term 
$r\in\psi\ [x\in A, w\in \phi]$ in \mtt.
\end{roll}
The action of the functor on arrows in \CM is given by
substitution.

The  doctrine 
$$\Fs:\SM\op\longrightarrow \mathsf{InfSl}$$
is defined on a closed set $A$ as the partial order consisting of
\begin{roll}
\item[equivalence classes $\ec{\phi\prps [x\in A]}$] of
\dfn{small} propositions $\phi(x)\prps [x\in A]$ depending on $A$,
{\it i.e.} propositions closed only under quantification on sets, with
respect to equiprovability in \mtt,
\item[$\ec{\phi\prps [x\in A]}\leq\ec{\phi'\prps [x\in A]}$] 
if there is a proof-term of
$r\in \phi'\ [x\in A, w\in \phi]$ in \mtt.
\end{roll}
The action of the functor on arrows in \SM is given by
substitution.

\begin{prop}
\label{mttd}
The doctrines $\Fm:\CM\op\longrightarrow\mathsf{InfSl}$ 
and $\Fs:\SM\op\longrightarrow\mathsf{InfSl}$ are existential
elementary, implicational and universal, with full weak comprehensions and comprehensive weak
equalizers. Moreover, the doctrine $\Fs$ is weakly cartesian
closed.
\end{prop}

\begin{proof}
The proof for each is analogous to that for Martin-L{\"o}f's type
theory in
propositions~\ref{mlttd} and \ref{mlbc}. 
\end{proof}

\begin{remark}
Similarly to~\ref{mlttequiv},
there is an essential equivalence between the completion $\Q{\Fm}$ of
the doctrine $\Fm$ and the model of extensional collections in
\cite{m09}, and there is another between the completion $\Q{\Fs}$ of
the doctrine $\Fs$ and the model of extensional sets in \cite{m09}. 

Then, if one recalls that \emtt is
the extensional version of \mtt with the addition of effective quotients,
proof-irrelevance of propositions 
and extensional equality of functions, it is easy to see that
there is an obvious interpretation of
\emtt in the doctrine 
$\Q{\Fm}$ provided that a choice of the structure is given.

Actually, by means of logical constructors of \emtt, following  the paradigm proposed in \cite{LawvereF:adjif,jacobbook},
we can describe the logical theories modelled by most of the doctrines considered in this paper.

For example,
 the notion of elementary existential doctrine  with full
 comprehensions and comprehensive equalizers is the necessary
 structure to interpret the set-theoretic fragment of \emtt  where sets just include the singleton set,
 binary products and indexed sums of a set with a proposition, and
 whose propositions are just closed  under
 conjunctions, extensional equality, existential quantifiers and the truth constant. Then the extension with \emtt-quotients  captures the logical theory behind
 existential doctrines in \QH.
\end{remark}

\begin{remark}
\label{notex}
The quotient model $\ct{Q}_{\Fs}$ of $\Fs$ for the fragment
of \mtt in \cite{iac} provides a
genuine example that the elementary quotient completion 
is not an exact completion.

Indeed, consider that in an exact, locally
cartesian closed category, (AUC) always holds, as shown in
\cite{tumscs}.
Now, if $\ct{Q}_{\Fs}$ were exact, then in the original doctrine
$\Fs$ (AUC) from any object to any other would hold by \ref{uac}.
But \cite{iac} shows that (AUC) from the natural numbers to the
natural numbers does not hold.

Knowing that in \CTT (AUC) from an object to another
does not hold in general, see \cite{streicher}, with the
same argument 
as above we get another example of elementary quotient completion that
is not exact: take  $\ct{Q}_{F^{\CTT}}$ with $F^{\CTT}$ defined as $\Fs$
but within  \CTT.
\end{remark}

\section{Conclusions and future work}
The  notion of elementary quotient
completion developed in the previous sections was inspired by the need
to give an abstract presentation of the quotient model used in
\cite{m09}. This notion is more general than that of exact
completion. As remarked in \ref{exlexcom}, an exact
completion of a cartesian category with weak pullbacks is an instance
of the elementary quotient completion construction.
On the other hand, as remarked in \ref{notex}, there are 
elementary quotient completions which are not exact.

Relevant instances of elementary quotient completion 
are used to formalize mathematics within an intensional
type theory. They are applied  to turn a theory 
with {\it intensional} and {\it weak} structures
 into one with {\it extensional} and {\it strong} ones.
For example, the category
$\ct{Q}_{\Fs}$ of extensional sets over the intensional theory
\mtt, as well as the category $\ct{Q}_{\F}$ of setoids over
Martin-L{\"o}f's type theory, are cartesian closed, while the
corresponding categories of sets of \mtt and of Martin-L{\"o}f's 
type theory are not. Moreover
 the doctrines $\F$, $\Fs$ and $\Fm$
associated to \mtt and to Martin-L{\"o}f's
type theory satisfy only 
weak comprehension,  because their
propositions may be equipped with more
than one proof.

But it is worth mentioning that to view an elementary quotient
completion, as that over \mtt in \cite{m09}, as a model of an
extensional theory as \emtt in \cite{m09}, one needs to solve a
coherence problem. 
Indeed, the base category of an elementary quotient completion, as
well as of an exact completion in \cite{CarboniA:freecl}, 
does not provide an explicit construction for the categorical
structure used to interpret logical sorts.

In \cite{m09} we solved the problem by using an {\it ad hoc} coherence
theorem to select the model structure necessary to interpret
the syntax of  \emtt in the elementary quotient completion over \mtt. Since 
that theorem
heavily relies on the fact that the syntax of \mtt is defined
inductively, a possible direction of further investigation
 is to explore how to interpret
a logical theory corresponding to doctrines in \QH in (the doctrine
 $\Q{P}$ of) an elementary quotient completion of a doctrine $P$ in  \CH.

In future work we shall 
 investigate the case of completing an existential elementary doctrine
with respect to quotients with a more relaxed notion of arrow and
 compare it with the elementary quotient completion introduced here.

\paragraph*{Acknowledgements} The authors would like to acknowledge
Giovanni Sambin and Ruggero Pagnan for very many, useful and intensive
discussions on the subject. Discussions with Fabio Pasquali and
Lorenzo Malatesta were also very useful for the second author.

\bibliographystyle{plain}
\bibliography{MaiettiRosolini_lu}
\end{document}